\pdfoutput=1
\documentclass[final,US,11pt,reqno,bibtex,nopagebackref,notcite,notref]{amsart}
\usepackage{liao}
\usepackage{comment}

\DeclareMathOperator{\tw}{tw}

\DeclareMathOperator{\pbw}{pbw}
\DeclareMathOperator{\pr}{pr}

\DeclareMathOperator{\Hom}{Hom}
\DeclareMathOperator{\RHom}{RHom}
\DeclareMathOperator{\coHom}{coHom}
\DeclareMathOperator{\Ext}{Ext}

\DeclareMathOperator{\End}{End}
\DeclareMathOperator{\Der}{Der}
\DeclareMathOperator{\Ber}{Ber}

\DeclareMathOperator{\coDer}{coDer}
\DeclareMathOperator{\tot}{tot}
\DeclareMathOperator{\id}{id}

\DeclareMathOperator{\CE}{\scriptscriptstyle{CE}}
\DeclareMathOperator{\ad}{ad}
\DeclareMathOperator{\td}{td}
\DeclareMathOperator{\at}{at}
\DeclareMathOperator{\Todd}{Td}

\DeclareMathOperator{\dec}{\textsf{d\'ec}}
\DeclareMathOperator{\rev}{rev}
\DeclareMathOperator{\sym}{sym}
\DeclareMathOperator{\hkr}{hkr}
\DeclareMathOperator{\poly}{poly}

\newcommand{\ssym}{\widetilde{\sym}}

\newcommand{\ZZ}{\mathbb{Z}} 
 
\newcommand{\RR}{\mathbb{R}} 
\newcommand{\CC}{\mathbb{C}} 
\newcommand{\cC}{\mathcal{C}}

\newcommand{\cU}{\mathcal{U}}

\newcommand{\KK}{\Bbbk}
\newcommand{\into}{\hookrightarrow}
\newcommand{\onto}{\twoheadrightarrow}
\newcommand{\xto}[1]{\xrightarrow{#1}}
\newcommand{\dual}{^{\vee}}

\newcommand{\inv}{^{-1}}

\newcommand{\argument}{-}
\newcommand{\pair}[2]{\langle \, #1 \mid #2 \,  \rangle}
\newcommand{\hcoh}{\mathbb{H}}

\newcommand{\frakg}{\mathfrak g}
\newcommand{\frakh}{\mathfrak h}
\newcommand{\grHH}{\mathrm{HH}_{\scriptscriptstyle{\oplus}}}
\newcommand{\grHHcx}{\mathrm{Hoch}_{\scriptscriptstyle{\oplus}}}
\newcommand{\pdHH}{\mathrm{HH}_{\scriptscriptstyle{\prod}}}
\newcommand{\pdHHcx}{\mathrm{Hoch}_{\scriptscriptstyle{\prod}}}
\newcommand{\HHcx}{\mathrm{Hoch}}

\newcommand{\hochschild}{d_{\mathcal{H}}}
\newcommand{\gerstenhaber}[2]{\llbracket #1, #2 \rrbracket}
\newcommand{\schouten}[2]{[ #1, #2 ]_S}
\newcommand{\Tpoly}{{_{\scriptscriptstyle{\oplus}}}{\mathcal{T}}_{\poly}}
\newcommand{\Dpoly}{{_{\scriptscriptstyle{\oplus}}}{\mathcal{D}}_{\poly}}

\newcommand{\pTpoly}{{\mathcal{T}}_{\poly}}

\newcommand{\op}[1]{{#1}^{\mathrm{op}}}

\DeclareMathOperator{\Cone}{Cone}
\newcommand{\cone}[1]{\Cone({#1})}

\newcommand{\suspend}{\mathrm{s}} 
\newcommand{\dsuspend}{\mathfrak{s}} 
\newcommand{\Tcoder}[1]{D_{#1}} 

\newcommand{\degree}[1]{|{#1}|}

\newcommand{\conv}{\star}

\newcommand{\sections}[1]{\Gamma\left(#1 \right)}
\newcommand{\tangent}[1]{T_{#1}}

\newcommand{\cM}{\mathcal{M}}

\newcommand{\cD}{\mathcal{D}}

\newcommand{\HmtpA}{\ssym}
\newcommand{\HmtpB}[1]{\mathcal{J}_{#1}}
\newcommand{\HmtpC}{\phi}
\DeclareMathOperator{\sgn}{sgn}

\usepackage{scalerel}

\newcommand{\action}{{\,\scriptscriptstyle\blacklozenge\,}}
\newcommand{\Sgaction}{{\lrcorner\,}}
\newcommand{\Sgvaction}{{\llcorner\,}}

\newenvironment{pfdsKellerThm}{\textit{Proof of Theorem~\ref{thm:d.s.KellerThm}.}}{\hfill$\square$}

\title{Keller admissible triples and Duflo theorem}

\author{Hsuan-Yi Liao
}
\address{Department of Mathematics, National Tsing Hua University}
\email{hyliao@math.nthu.edu.tw} 

\author{
Seokbong Seol}
\address{School of Mathematics, Korea Institute for Advanced Study}
\email{azuredream89@kias.re.kr} 

\thanks{Research partially supported by NSF grants DMS-1707545 and DMS-2001599, KIAS Individual Grant MG072801 and MG090801, and MOST/NSTC Grant 110-2115-M-007-001-MY2.}

\begin{document}

\begin{abstract}
 The present paper is devoted to the study of Keller admissible triples. We prove that a Keller admissible triple induces an isomorphism of Gerstenhaber algebras between  the Hochschild cohomologies of direct-sum type of the pair of differential graded algebras bound to one another by the admissible triple. As an application, we give a new concrete proof of the Duflo--Kontsevich theorem for finite-dimensional Lie algebras. 
 \end{abstract}

\maketitle

\tableofcontents

\section*{Introduction}

Let $A$ and $B$ be dg (i.e.\ differential graded) algebras over a field $\KK$. It is a theorem of Keller \cite{keller2003derived} that, if there exists a dg $A$-$B$-bimodule $X$ and the induced canonical morphisms 
\begin{gather*}
 A \to \RHom_{\op B}(X,X), \\
\op B \to \RHom_{A}(X,X)
\end{gather*}
are quasi-isomorphisms, then the Hochschild cohomologies of $A$ and $B$ are isomorphic as Gerstenhaber algebras. In most of the literature, including \cite{keller2003derived}, the Hochschild cohomology of $A$ is defined by derived functors and is equal to the cohomology $\pdHH^\bullet(A)$ of the \emph{product-total complex} 
$$
\Big(\prod_{p+r = \bullet} \Hom^{r}(A^{\otimes p}, A), \, \hochschild + \partial_A \Big).
$$ 
However, in this paper, we are interested in the cohomology $\grHH^\bullet(A)$ of the \emph{sum-total complex} 
$$
\Big(\bigoplus_{p+r = \bullet} \Hom^{r}(A^{\otimes p}, A), \, \hochschild + \partial_A \Big).
$$  
The second kind of Hochschild cohomology plays a significant role in formality theorem and Duflo--Kontsevich-type theorem for dg manifolds \cite{MR3754617,MR2986860, MR2816610, 1998math.....12009S}. 
Here, a dg manifold (i.e. differential graded manifold) means a $\ZZ$-graded manifold equipped with a homological vector field. 
The notion of dg manifolds has recently drawn attention for its connections with various fields in mathematics including Lie theory, homotopy Lie algebras, foliations, higher Lie algebroids and derived differential geometry \cite{2020arXiv200601376B,MR2840338, MR2819233, MR3877426, MR4504932, MR4164730, MR4352591, MR3313214, MR2768006, MR2971727}.  
In addition to formality theorem and Duflo--Kontsevich-type theorem for dg manifolds, the second type Hochschild cohomology is also important in matrix factorization \cite{MR2931331}.

The sum Hochschild cohomology behaves significantly differently from the product Hochschild cohomology and is not well-behaved under homotopy equivalences. See Example~\ref{ex:SumProdHH} (and Example~\ref{ex:KellerAdmissibleSumHoch}) for an instance of a dg algebra $A$ whose Hochschild cohomologies $\pdHH^{\bullet}(A)$ and $\grHH^{\bullet}(A)$ are different. 
In this paper, Hochschild complexes/cohomologies means the \emph{sum} Hochschild complexes/cohomologies unless otherwise stated. 

Inspired by Keller's theorem, we introduce a notion of \emph{Keller admissible triples} and prove a sum Hochschild cohomology version of Keller's result.

For us, a Keller admissible triple is a triple $(A,X,B)$, where $A$ and $B$ are dg algebras, and $X$ is a dg $A$-$B$-bimodule, such that the action maps
\begin{gather*}
\rho_A: (A,d_A) \to  \big(\Hom_{\op B}(X,X), \partial_X\big) , \\
\rho_B: (\op B,d_B) \to  \big(\Hom_A(X,X), \partial_X \big)
\end{gather*}
are weakly cone-nilpotent quasi-isomorphisms, which we define in Section~\ref{sec:WCN}, and the partial Hochschild complexes 
\begin{gather*}
0 \to \Hom_{\op B}(X,X) \into \Hom(X,X) \xto{{\hochschild^X}_R} \Hom(X \otimes B, X) \xto{{\hochschild^X}_R} \Hom(X \otimes B^{\otimes 2}, X) \xto{{\hochschild^X}_R} \cdots   \\
0 \to \Hom_{A}(X,X) \into \Hom(X,X) \xto{{\hochschild^X}_L} \Hom(A \otimes X, X) \xto{{\hochschild^X}_L} \Hom(A^{\otimes 2} \otimes X, X) \xto{{\hochschild^X}_L} \cdots 
\end{gather*}
are exact. Essentially, such a triple can be considered as a kind of Morita equivalence between the dg algebras $A$ and $B$.

For a dg $A$-$B$-bimodule $X$, Keller constructed a dg category (with two objects) and studied its Hochschild cohomology $\pdHH^{\bullet}(X)$ which can be obtained as the cohomology of a product-total complex $\pdHHcx^{\bullet}(X)$ \cite{keller2003derived,MR2275593,MR2628794}.
Instead of the product-total complex, we consider the sum Hochschild complex $\grHHcx^\bullet(X)$ which is equipped with a pair of natural projections $\pi_A:\grHHcx^\bullet(X) \onto \grHHcx^\bullet(A)$ and $\pi_B:\grHHcx^\bullet(X) \onto \grHHcx^\bullet(B)$. Our main result is the following:
\begin{trivlist}
\item {\bf Theorem A.} {\it For a Keller admissible triple $(A,X,B)$, the pair of natural projections $\pi_A$ and $\pi_B$ induce a pair of isomorphisms of Gerstenhaber algebras $\grHH^{\bullet}(A)\cong \grHH^{\bullet}(X)\cong \grHH^{\bullet}(B)$ between the Hochschild cohomologies.}
\end{trivlist}

Theorem~A provides a way to construct an explicit isomorphism between the Hochschild cohomologies of ``Morita equivalent'' dg algebras. As an application, we obtain, in Section~\ref{sec:KellerTrip-LieAlg} and Section~\ref{sec:KD}, an alternative proof of the Duflo--Kontsevich theorem for finite-dimensional Lie algebras. 

Given a finite-dimensional Lie algebra $\frakg$, the symmetrization map $\pbw:S\frakg \to \cU\frakg$ from the symmetric algebra $S\frakg$ to the universal enveloping algebra $\cU\frakg$ is an isomorphism of vector spaces, called Poincar\'e--Birkhoff--Witt isomorphism.
It induces a vector space isomorphism $\pbw: (S\frakg)^\frakg \to (\cU\frakg)^\frakg$ between the subspaces of $\frakg$-invariants. This isomorphism does not respect the natural associative multiplications on $(S\frakg)^{\frakg}$ and $(\cU\frakg)^{\frakg}$, but it can be modified by the square root of Duflo element so as to obtain an isomorphism of associative algebras. The Duflo element $J \in \widehat{S}(\frakg\dual)$ is the formal power series on $\frakg$ defined by $J(x)= \det\left( \frac{1-e^{-\ad_x}}{\ad_x}\right)$, for all $x \in \frakg$. Its square root $J^{1/2}$ acts on $(S\frakg)^\frakg$ in the manner of a formal differential operator. 
A remarkable theorem due to Duflo \cite{MR444841} asserts that the composition $\pbw\circ J^{1/2}:(S\frakg)^\frakg \to (\cU\frakg)^\frakg$ is an isomorphism of associative algebras. This theorem generalizes a fundamental theorem of Harish--Chandra about the center of the universal enveloping algebra of semi-simple Lie algebras. 

In \cite{MR2062626}, Kontsevich proposed a new proof of Duflo's theorem as an application of his quantization formulas. More precisely, Kontsevich recovered the Duflo isomorphism as the tangent map of his formality morphism at the Lie--Poisson structure on $\frakg\dual$ regarded as a Maurer--Cartan element in the dgla of polyvector fields $\pTpoly^\bullet(\frakg\dual)$ on $\frakg\dual$. This approach led to the Duflo--Kontsevich theorem: the map 
\begin{equation}\label{eq:KDisoLieAlg}
\pbw \circ J^{1/2}: H_{\CE}^\bullet(\frakg, S\frakg) \to H_{\CE}^\bullet(\frakg, \cU\frakg)
\end{equation}
is an isomorphism of graded algebras. The classical Duflo theorem is the isomorphism of cohomologies in degree $0$. A detailed proof of the Duflo--Kontsevich theorem was given by Pevzner--Torossian \cite{MR2085348}.

Shoikhet \cite{1998math.....12009S} conjectured that a Duflo--Kontsevich-type theorem should hold for all finite dimensional dg manifolds.  
Shoikhet's conjecture was proved two decades later by the first author jointly with Sti\'enon and Xu \cite[Theorem~4.3]{MR3754617}: 
for a finite dimensional dg manifold $(\cM,Q)$, the map 
\begin{equation}\label{eq:KDisoforDGmfd}
\hkr \circ \Todd^{1/2}_{(\cM,Q)}: \hcoh^{\bullet} \big( \Tpoly(\cM), \schouten{Q}{\argument} \big) \xto{}   \hcoh^{\bullet}\big( \Dpoly(\cM), \hochschild + \gerstenhaber{Q}{\argument} \big)
\end{equation}
is an isomorphism of graded algebras. 
Here, $\hkr$ is the Hochschild--Kostant--Rosenberg map, and $\Todd_{(\cM,Q)}^{1/2}$ is the square root of the Todd class acting on $\Tpoly(\cM)$ by contraction. 

Shoikhet also suggested that the Duflo--Kontsevich theorem \eqref{eq:KDisoLieAlg} could be recovered by applying \eqref{eq:KDisoforDGmfd} to the dg manifold $(\frakg[1],d_\frakg)$, where the Chevalley--Eilenberg differential $d_\frakg \in \Tpoly^1(\frakg[1])$ is considered as a homological vector field on $\frakg[1]$. 
Indeed, by doing so, we obtain the isomorphism of graded algebras 
$$
\hkr\circ \Todd_{\frakg[1]}^{1/2}: \hcoh^{\bullet} \big( \Tpoly(\frakg[1]), \schouten{d_\frakg}{\argument} \big) \xto{}   \hcoh^{\bullet}\big( \Dpoly(\frakg[1]), \hochschild + \gerstenhaber{d_\frakg}{\argument} \big)
$$ 
from the cohomology of polyvector fields to the cohomology of polydifferential operators on $\frakg[1]$.
The cohomology $\hcoh^{\bullet} \big( \Tpoly(\frakg[1]), \schouten{d_\frakg}{\argument} \big)$ is naturally identified with the Chevalley--Eilenberg cohomology $H_{\CE}^\bullet(\frakg, S\frakg)$, while the cohomology $\hcoh^{\bullet}\big( \Dpoly(\frakg[1]), \hochschild + \gerstenhaber{d_\frakg}{\argument} \big)$ coincides with the Hochschild cohomology $\grHH^\bullet (S(\frakg[1])\dual,d_\frakg)$. 

In order to derive \eqref{eq:KDisoLieAlg}, we deduce an explicit isomorphism $\grHH^\bullet(S(\frakg[1])\dual,d_\frakg) \cong \grHH^\bullet(\cU\frakg)$ of Gerstenhaber algebras from the Keller admissible triple $\big(\cU \frakg, \ \cU\frakg  \otimes S( \frakg[1]) , \  S( \frakg[1])\dual\big)$, where $S( \frakg[1])\dual$, endowed the differential $d_\frakg$, is the Chevalley--Eilenberg algebra of $\frakg$. 
Composition with the standard Cartan--Eilenberg isomorphism $\grHH^\bullet(\cU\frakg) \cong H_{\CE}^\bullet(\frakg,\cU\frakg)$ described in \eqref{eq:HH(Ug)isoHce(Ug)} yields the desired isomorphism of graded algebras 
\[\Phi_D:\hcoh^{\bullet}\big( \Dpoly(\frakg[1]), \hochschild + \gerstenhaber{d_\frakg}{\argument} \big)\cong \grHH^\bullet (S(\frakg[1])\dual,d_\frakg) \xto\cong H_{\CE}^\bullet (\frakg, \cU\frakg).\]
Finally, we prove the following theorem, which implies that the Duflo--Kontsevich map \eqref{eq:KDisoLieAlg} is an isomorphism of graded algebras.

\begin{trivlist}
\item {\bf Theorem B.} {\it Given any finite-dimensional Lie algebra $\frakg$, the diagram
$$
\begin{tikzcd}
\hcoh^{\bullet} \big( \Tpoly(\frakg[1]), \schouten{d_\frakg}{\argument} \big) \arrow[rr, "\hkr\circ \Todd_{\frakg[1]}^{1/2}"] \arrow[d, "\Phi_T"',"\cong"]  & &   \hcoh^{\bullet}\big( \Dpoly(\frakg[1]), \hochschild + \gerstenhaber{d_\frakg}{\argument} \big)  \arrow[d, "\Phi_{D}","\cong"']  \\
  H_{\CE}^\bullet(\frakg, S\frakg ) \arrow[rr, "\pbw \circ J^{1/2}"]  & &  H_{\CE}^\bullet(\frakg, \cU \frakg) 
\end{tikzcd}
$$
commutes. Here, $\Phi_T$ is the natural identification, while $\Phi_D$ is the identification induced by the Keller admissible triple $\big(\cU \frakg, \ \cU\frakg  \otimes S( \frakg[1]) , \  S( \frakg[1])\dual \big)$.
}
\end{trivlist}

In the process of proving the above theorem, we verify that $\Phi_{D} \circ \hkr = \pbw \circ\, \Phi_T$, i.e.\ the Hochschild--Kostant--Rosenberg map for the graded manifold $\frakg[1]$ is identified to the Poincar\'e--Birkhoff--Witt map by way of the Keller admissible triple $\big(\cU \frakg, \ \cU\frakg  \otimes S( \frakg[1]) , \  S( \frakg[1])\dual \big)$.

In fact, by a method of twisting cochains \cite{MR4237031,MR2954392} or by spectral sequences \cite{MR2816610}, it is possible to construct a quasi-isomorphism $\grHHcx^\bullet (S(\frakg[1])\dual,d_\frakg) \to C_{\CE}^\bullet (\frakg, \cU\frakg)$ which induces the isomorphism $\Phi_D$ on the cohomology level. 
In \cite{MR4237031}, based on a theorem of Lowen--Van den Bergh \cite{MR4537334}, Keller proved that there exists an isomorphism of Gerstenhaber algebras between the coHochschild cohomology of the dg coalgebra $\Lambda \frakg$ and the Hochschild cohomology of $ U\frakg$ in terms of Koszul--Moore duality. It would be interesting to explore the relation between Keller admissible triples and Koszul--Moore duality.

Most theorems in the present paper are derived from explicit formulas obtained by computations in coalgebras. For completeness, we summarize the relevant coalgebra techniques in the appendices.

\subsection*{Notations and conventions}

Throughout the paper, the symbol $\KK$ denotes either of the fields $\RR$ or $\CC$, and graded means $\ZZ$-graded. Unless specified otherwise, dg algebra means dg algebra over $\KK$, $\otimes$ means $\otimes_{\KK}$, and $\Hom$ means $\Hom_{\KK}$.

Let $A$ be a graded ring, and let $V$, $W$ be two left (respectively, right) graded $A$-modules. 
The space $\Hom_A(V,W) = \Hom_A^\bullet(V,W)$ (respectively, $\Hom_{\op A}(V,W) = \Hom_{\op A}^\bullet(V,W)$) of morphisms of left (respectively, right) $A$-modules from $V$ to $W$ is naturally $\ZZ$-graded: the symbol $\Hom^r_{A}(V,W)$ (respectively, $\Hom^{r}_{\op A}(V,W)$)  denotes the space of morphisms of left (respectively, right) $A$-modules of degree $r$ from $V$ to $W$. We denote by $V[i]$ the graded $A$-module with the homogeneous components $(V[i])^j = V^{i+j}$.   

Given a homogeneous element $x$ in a graded vector space $V=\bigoplus_{k\in \ZZ}V^{k}$, we write $|x|$ to denote the degree of $x$. Thus $|x|=d$ means that $x\in V^{d}$.

For $x \in V$ and $\xi \in V\dual$, we denote $\pair{\xi}{x}:=\xi(x)$ and $\pair{x}{\xi} := (-1)^{|x||\xi|}\pair{\xi}{x}$. The pairing is extended to a pairing of tensor algebras $\pair\argument\argument : TV \times TV\dual  \to \KK$ by 
$$
\pair{x_1 \otimes \cdots \otimes x_p}{\xi_1 \otimes \cdots \otimes \xi_q} := 
\begin{cases}
(-1)^{\sum_{i=1}^p\sum_{j=i+1}^p|\xi_i| |x_j|}\ \pair{x_1}{\xi_1} \cdots \pair{x_p}{\xi_p}, & \text{if } p = q, \\
0, & \text{if } p \neq q;
\end{cases}
$$
and to a pairing of symmetric algebras $\pair\argument\argument: SV \times SV\dual \to \KK$ by 
$$
\pair{x_1 \odot \cdots \odot x_p}{\xi_1 \odot \cdots \odot \xi_q} := \sum_{\sigma \in S_q} \varepsilon \cdot \pair{x_1 \otimes \cdots \otimes x_p}{\xi_{\sigma(1)} \otimes \cdots \otimes \xi_{\sigma(q)}},
$$
where $\varepsilon = \pm 1$ is the number such that 
$\xi_1 \odot \cdots \odot \xi_q = \varepsilon \cdot \xi_{\sigma(1)} \odot \cdots \odot \xi_{\sigma(q)}$ in $SV\dual$.
Similarly, we also have the pairings $\pair\argument\argument : TV\dual \times TV \to \KK$ and $\pair\argument\argument : SV\dual \times SV \to \KK$.

We denote by $\suspend:V[1] \to V$ the degree-shifting map of degree  $+1$ and by  $\dsuspend:V \to V[1]$ the degree-shifting map of degree $-1$.

\subsection*{Acknowledgments} 
We would like to thank Estanislao Herscovich, Bernhard Keller, Mathieu Sti\'enon, Ping Xu and Bin Zhang for fruitful discussions and useful comments. We are indebted to Leonid Positselski for pointing out an error in the first preprint of the paper. 
We also wish to thank Boris Shoikhet for suggesting the approach to the Duflo theorem via Keller admissible triples. 
Seol is grateful to the Korea Institute for Advanced Study for its hospitality in 2020.

\section{Hochschild complexes of differential graded algebras} \label{sec:HHcxDGAlg}

In this section, we recall the explicit formulas of the Gerstenhaber bracket and the cup product on the Hochschild cochain complex of a differential graded algebra. 
These structures were first introduced by Gerstenhaber \cite{MR161898} for an ungraded ring, and were generalized to dg algebras by various authors in slightly different ways. 
The formulas in this section are obtained by composing the formulas in \cite{MR1261901} with proper degree-shifting maps. More details can be found in Appendix~\ref{appendix:HHcochains}.

Let $(A,d_A)$ and $(B,d_B)$ be dg algebras. Recall that a \emph{dg $A$-$B$-bimodule} $(X,d_X)$ is a graded $A$-$B$-bimodule $X$ together with a differential $d_X$ such that 
$$d_X(a x  b) = d_A(a)xb + (-1)^{|a|}a d_X(x) b + (-1)^{|a|+|x|} a x d_B(b),$$ for  $a \in A,\ x \in X,\ b \in B$.

Let  $(M,d_M)$ be a dg $A$-$A$-bimodule. 
A \textbf{Hochschild cochain} of degree $(p,r)$ of $A$ with values in $M$ is an element in 
$$
\HHcx^{p,r}(A,M):=  \Hom^{r}\big(A^{\otimes p},M \big).
$$
The \textbf{Hochschild cochain complex} of $A$ with values in $M$ is the space 
\begin{equation*}\label{eq:SumTotCxHH}
\grHHcx^\bullet(A,M) := \bigoplus_{p+r = \bullet} \HHcx^{p,r}(A,M)
\end{equation*} 
together with the differential $\hochschild + \partial : \grHHcx^\bullet(A,M) \to \grHHcx^{\bullet +1}(A,M)$, where  
\begin{equation}
\begin{split}
\hochschild(f)(a_0,\cdots, a_{p})  & := (-1)^{(p+r-1) + r\degree{a_0}} a_0 f(a_1, \cdots, a_{p}) \\
& \quad  + \sum_{i=0}^{p-1} (-1)^{p+r+i} f(a_0, \cdots, a_i a_{i+1}, \cdots, a_{p}) \\
&\quad + (-1)^{r} f(a_0, \cdots, a_{p-1})a_{p}
\end{split}
\end{equation} 
and 
\begin{equation}\label{eq:DiffFromDG}
\partial(f)(a_1, \cdots, a_p)  := d_M f(a_1, \cdots, a_p) - (-1)^r \sum_{i=1}^p (-1)^{|a_1| + \cdots + |a_{i-1}|} f ( a_1, \cdots, d_A a_i, \cdots, a_p),
\end{equation}
for $f \in \HHcx^{p,r}(A,M)$, $a_0, \cdots, a_{p} \in A$. 
The cohomology of $\big(\grHHcx^\bullet(A,M), \hochschild + \partial \big)$ is called the \textbf{Hochschild cohomology} of $(A,d_A)$ with values in $(M,d_M)$, denoted by $\grHH^\bullet(A,M)$. In the case $(M,d_M)=(A,d_A)$, we denote $\partial_A = \partial$, $\grHHcx^\bullet(A,d_A)=  \grHHcx^\bullet(A,A)$ and $\grHH^\bullet(A,d_A)=\grHH^\bullet(A,A)$. We will omit $d_A$ in the notations if the differential is clear from the context.

\begin{remark}
In the literature (see e.g. \cite{MR2275593,keller2003derived,MR3941473}),  the Hochschild cohomology of a dg algebra $(A,d_A)$ is defined by derived functors which can be computed by the \emph{product-total complex} of the double complex $\big(\HHcx(A), \hochschild, \partial_A \big)$,
\begin{equation*}\label{eq:ProdTotCxHH}
\pdHHcx^\bullet(A) := \prod_{p+r = \bullet} \HHcx^{p,r}(A)
\end{equation*} 
whose cohomology $\pdHH^\bullet(A)$ is different from $\grHH^\bullet(A)$ in general.  
\end{remark}

The sum Hochschild cohomology $\grHH^{\bullet}(A)$ is sometimes referred as the compactly supported Hochschild cohomology \cite{MR2931331}.

\begin{example}\label{ex:SumProdHH}
Let  $(A,d_A)$ be the dg algebra $(S(\KK[1])\dual, 0) \cong (\KK[x]/(x^2),0) \cong(\KK\oplus \KK x, 0)$, where $x$ is a formal variable of degree one. It is straightforward to show that 
\[ \grHH^{0}(A)\cong \bigoplus_{n=0}^{\infty} \KK \neq \prod_{n=0}^{\infty} \KK\cong \pdHH^{0}(A). \]
\end{example}

Let $f \in \HHcx^{p_1,r_1}(A)$ and $g \in \HHcx^{p_2,r_2}(A)$. 
The \emph{cup product} $f \cup g \in \HHcx^{p_1+p_2,r_1+r_2}(A)$ is defined by the formula 
\begin{equation}\label{eq:CupProd}
(f \cup g)(a_1, \cdots, a_{p_1+p_2}) := (-1)^{p_1p_2+ r_2(|a_1|+\cdots + |a_{p_1}|+p_1)} f(a_1, \cdots, a_{p_1}) \cdot g(a_{p_1+1},\cdots, a_{p_1+p_2}),
\end{equation}
for $a_1,\cdots, a_{p_1+p_2} \in A$.

Let $f \circ_i g \in \HHcx^{p_1+p_2-1,r_1+r_2}(A)$ be the \emph{$i$-th composition} 
\begin{equation}\label{eq:i-thGerComposition}
f \circ_i g :=  f \circ (\id^{\otimes i-1} \otimes g \otimes \id^{\otimes p_1 -i}).
\end{equation}
The \emph{Gerstenhaber bracket} $\gerstenhaber{f}{g} \in \HHcx^{p_1+p_2-1, r_1+r_2}(A)$ of $f$ and $g$ is 
\begin{equation}\label{eq:GerBracket}
\begin{split}
\gerstenhaber f g &:= \sum_{i=1}^{p_1} (-1)^{ (p_1-1)r_2 + (i-1)(p_2-1)} \, f \circ_i g \\
&\qquad \qquad\qquad - (-1)^{(p_1+r_1-1)(p_2+r_2-1)} \sum_{j=1}^{p_2} (-1)^{(p_2-1)r_1 + (j-1)(p_1-1)} \,  g \circ_j f.
\end{split}
\end{equation}  
One can show that 
$$
\hochschild = \gerstenhaber{\mu_A}{\argument}, \qquad \partial_A = \gerstenhaber{d_A}{\argument},
$$
where the multiplication $\mu_A: a \otimes b \mapsto a \cdot b$ in $A$ is considered as a Hochschild cochain in $\HHcx^{2,0}(A)$.

The following proposition can be shown by a direct computation as in \cite{MR161898}. 

\begin{proposition}\label{prop:Hochfordga}
Let $(A,d_A)$ be a dg algebra.  
\begin{itemize}
\item[(i)]
The shifted Hochschild cochain complex $\big(\grHHcx^\bullet(A)[1], \hochschild + \partial_A, \gerstenhaber{\argument}{\argument} \big)$ together with the Gerstenhaber bracket $\gerstenhaber{\argument}{\argument}$ is a differential graded Lie algebra. 
\item[(ii)]
The Hochschild cochain complex $\big(   \grHHcx^\bullet(A), \hochschild + \partial_A , \cup  \big)$ together with the cup product $\cup$ is a differential graded algebra.
\end{itemize}
Furthermore, the Hochschild cohomology $\grHH^{\bullet}(A)$ is a Gerstenhaber algebra.
\end{proposition}

\section{Hochschild complexes of differential graded bimodules}

Let $(A,d_A)$ and $(B,d_B)$ be dg algebras, and $(X,d_X)$ be a dg $A$-$B$-bimodule. 
In \cite{keller2003derived}, Keller constructed a dg category of two objects from $(X,d_X)$ and studied its Hochschild cohomology which can be computed by a product-total complex $\pdHHcx^\bullet(X)$. 
In this section, we consider the dg algebra $A\ltimes X \rtimes B$ and introduce the (sum) Hochschild complex $\grHHcx^\bullet(X)$ of $X$ as a subcomplex of the Hochschild complex $\grHHcx^\bullet(A\ltimes X \rtimes B)$. This is a sum analogue of the Keller's complex $\pdHHcx^\bullet(X)$.

Let $A\ltimes X \rtimes B$ be the dg algebra whose underlying graded vector space is the direct sum $A \oplus X \oplus B$, and whose multiplication and differential are defined, respectively, by the formulas
\begin{gather*}
(a_1+x_1+b_1) \cdot (a_2+x_2 +b_2):= a_1a_2 + (a_1x_2 + x_1 b_2) + b_1b_2, \\
d(a_1+x_1+b_1) := d_A(a_1) + d_X(x_1) + d_B(b_1),
\end{gather*} 
for $a_1,a_2 \in A$, $x_1,x_2 \in X$, $b_1,b_2 \in B$. 
It is straightforward to show that $A\ltimes X \rtimes B$ is a dg algebra.

We will use the notations 
\begin{gather*}
\HHcx^{p,q,r}(A,X,B):= \Hom^r(A^{\otimes p} \otimes X \otimes B^{\otimes q}, X), \\
\grHHcx^\bullet(A,X,B):= \bigoplus_{p+q+r +1 = \bullet} \HHcx^{p,q,r}(A,X,B), \\
\grHHcx^\bullet(X):= \grHHcx^\bullet(A) \oplus \grHHcx^\bullet(A,X,B) \oplus \grHHcx^\bullet(B).
\end{gather*}
The space $\grHHcx^\bullet(X)$ can be embedded into $\grHHcx^\bullet(A\ltimes X \rtimes B)$ as follows:
\begin{gather*}
\grHHcx^p(A) \into \grHHcx^p(A\ltimes X \rtimes B): f_A \mapsto i_A \circ f_A \circ \pr_A^{\otimes p}, \\
\grHHcx^q(B) \into \grHHcx^q(A\ltimes X \rtimes B): f_B \mapsto i_B \circ f_B \circ \pr_B^{\otimes q}, \\
\HHcx^{p,q,r}(A,X,B) \into \grHHcx^{p+q+r+1}(A\ltimes X \rtimes B): f_X \mapsto i_X \circ f_X \circ(\pr_A^{\otimes p} \otimes \pr_X \otimes \pr_B^{\otimes q}), 
\end{gather*}
where $i_A, i_B, i_X$ are the inclusions from $A, B, X$ into $A\ltimes X \rtimes B$, respectively, and $\pr_A,\pr_B,\pr_X$ are the projections from  $A\ltimes X \rtimes B$ onto $A, B, X$, respectively. We will omit $i$ and $\pr$ by abuse of notation.  
With this embedding, one can show that the subspace $\grHHcx^\bullet(X)$ is closed under the differential, Gerstenhaber bracket and cup product in $\grHHcx^\bullet(A\ltimes X \rtimes B)$.

\begin{proposition}\label{prop:HH(X)_dga/dgla_str}
Let $X$ be a dg $A$-$B$-bimodule. The subspace $\grHHcx^\bullet(X)$ is closed under the differential $\hochschild + \partial$, Gerstenhaber bracket $\gerstenhaber{\argument}{\argument}$ and cup product $\cup$ in $\grHHcx^\bullet(A\ltimes X \rtimes B)$.
\end{proposition}

In the rest of this section, we describe the differential, Gerstenhaber bracket and cup product on $\grHHcx^{\bullet}(X)$.

Let $\hochschild + \partial$ be the restriction of the differential of $\grHHcx^\bullet(A\ltimes X \rtimes B)$ to $\grHHcx^\bullet(X)$. Here, $\hochschild$ denotes the Hochschild differential, and $\partial$ denotes the differential induced by the dg structure of $A$, $B$ and $X$.  The Hochschild differential $\hochschild$ can be decomposed as 
$$
\hochschild = \hochschild^A + \hochschild^{AX} + {\hochschild^X}_L + {\hochschild^X}_R + \hochschild^{XB} + \hochschild^B
$$ 
which acts on $\grHHcx^\bullet(X)$ as in the following diagram:
$$
\begin{tikzcd}
&  \grHHcx^\bullet(A,X,B) \ar[loop,distance=1cm,"\hochschild^X = {\hochschild^X}_L + {\hochschild^X}_R"']
 & \\ 
 \ar[loop left,distance=0.7cm,"\hochschild^A"] \grHHcx^\bullet(A) \ar[ru,"\hochschild^{AX}"] & & \ar[lu,"\hochschild^{XB}"'] \grHHcx^\bullet(B) \ar[loop right,distance=0.7cm,"\hochschild^B"]
\end{tikzcd}
$$
where $\hochschild^A$ and $ \hochschild^B$ are the Hochschild differentials of $A$ and $B$, respectively, and the other components $\hochschild^{AX}, {\hochschild^X}_L, {\hochschild^X}_R , \hochschild^{XB}$ are described as follows.

Let $f_A  \in \HHcx^{p,r}(A),   f_B \in \HHcx^{q,r}(B), f_X \in \HHcx^{p,q,r}(A,X,B)$, $a_i \in A$, $x \in X$, and $b_j \in B$. We have 
\begin{gather*}
\hochschild^{AX}(f_A) \in \HHcx^{p,0,r}(A,X,B), \qquad \hochschild^{XB}(f_B) \in \HHcx^{0,q,r}(A,X,B), \\
{\hochschild^X}_L(f_X) \in \HHcx^{p+1,q,r}(A,X,B), \qquad {\hochschild^X}_R(f_X) \in \HHcx^{p,q+1,r}(A,X,B) 
\end{gather*}
which are defined by 
\begin{align*}
\hochschild^{AX}(f_A)\big(a_1, \cdots, a_{p}; x\big) & := (-1)^{r}f_A(a_1,\cdots, a_{p})\cdot x, \\
\hochschild^{XB}(f_B)\big(x;b_1,\cdots, b_q\big) & := (-1)^{q+r-1+r|x|} x\cdot f_B(b_1,\cdots, b_q), \\
 {\hochschild^X}_L(f_X)\big(a_0,\cdots, a_p; x\, ; b_1,\cdots, b_q\big)  &:= (-1)^{p+q+r +r|a_0|} a_0 \cdot f_X(a_1,\cdots, a_p; x\,;b_1,\cdots, b_q) \\
& \qquad + \sum_{i=0}^{p-1} (-1)^{p+q+r+i+1} \ f_X(a_0, \cdots, a_ia_{i+1}, \cdots, a_p; x\,; b_1, \cdots, b_q) \\
& \qquad \qquad + (-1)^{q+r+1}\  f_X(a_0, \cdots,  a_{p-1}; a_p \cdot x\,; b_1, \cdots, b_q), \\
 {\hochschild^X}_R(f_X)\big(a_1,\cdots, a_p; x\,; b_0,\cdots, b_q\big) & := (-1)^{q+r-1} f_X(a_1,\cdots, a_p; x \cdot b_0 \,;b_1,\cdots, b_q) \\
& \qquad + \sum_{j=0}^{q-1} (-1)^{q+r+j} \ f_X(a_1,  \cdots, a_p; x\,; b_0,\cdots, b_j b_{j+1}, \cdots, b_q) \\
& \qquad \qquad + (-1)^{r}\ f_X(a_1, \cdots,   a_p;  x\,; b_0, \cdots, b_{q-1}) \cdot b_q.
\end{align*}

For the component $\partial$, we have 
$$
\partial = \partial_A + \partial_B + \partial_X,
$$
where $\partial_A$ and $\partial_B$ are the differentials defined by \eqref{eq:DiffFromDG} on $A$ and $B$, respectively, and the Hochschild cochain $\partial_X(f_X) \in \HHcx^{p,q,r+1}(A,X,B)$ is defined by
$$
\partial_X(f_X) := d_X \circ f_X  \\ 
- (-1)^r f_X \circ \Big( \sum_{i=1}^p \id^{\otimes i-1} \otimes d_A \otimes \id^{\otimes p+q-i+1} + \id^{\otimes p} \otimes d_X \otimes \id^{\otimes q} + \sum_{j=1}^q \id^{\otimes p+j} \otimes d_B \otimes \id^{\otimes q-j}\Big).
$$ 
The cochain complex $\big(\grHHcx^\bullet(X), \hochschild + \partial \big)$ will be referred as the \textbf{Hochschild cochain complex} of the dg $A$-$B$-bimodule $X$.

Let $\circ_i$ be the $i$-th composition in $\grHHcx^\bullet(A\ltimes X \rtimes B)$, which is defined as in \eqref{eq:i-thGerComposition}. 
We have 
$$
f_A \circ_i f_X = f_A \circ_i f_B = f_B \circ_i f_A = f_B \circ_i f_X =0, 
$$
for any $i$, and
$$
f_X \circ_j f_A =0, \quad f_X \circ_k f_X'  = 0, \quad f_X \circ_l f_B =0,
$$
if $j>p$, $k\neq p+1$, $l<p+2$. 
Here, $f_X'$ is a Hochschild cochain in $\grHHcx^\bullet(A,X,B) \subset \grHHcx^\bullet(X)$.

Furthermore, it can be shown by \eqref{eq:CupProd} that
$$
f_A \cup f_B =0, \quad f_X \cup f_A =0, \quad f_X \cup f_X' =0, \quad f_B \cup f_A =0, \quad f_B \cup f_X =0,
$$
and 
$$
f_A \cup f_X, \, f_X \cup f_B \in \grHHcx^\bullet(A,X,B).
$$

As a consequence, we have the following 

\begin{proposition}\label{prop:ProjInclPreserveStr}
Let $\pi_A: \grHHcx^\bullet(X) \onto \grHHcx^\bullet(A)$, $\pi_B: \grHHcx^\bullet(X) \onto \grHHcx^\bullet(B)$ be the natural projections, and let $\iota_A: \grHHcx^\bullet(A) \into \grHHcx^\bullet(X)$, $\iota_B: \grHHcx^\bullet(B) \into \grHHcx^\bullet(X)$ be the natural inclusions. 
\begin{itemize}
\item[(i)]
The projections $\pi_A$ and $\pi_B$ are cochain maps. 
\item[(ii)]
The projections $\pi_A, \pi_B$ and the inclusions $\iota_A, \iota_B$ respect the compositions $\circ_i$ and thus preserve the Gerstenhaber brackets. 
\item[(iii)]
The projections $\pi_A, \pi_B$ and the inclusions $\iota_A, \iota_B$ preserve the cup products.
\end{itemize}
\end{proposition}

Note that the inclusions $\iota_A$ and $\iota_B$ are \emph{not} cochain maps.

\section{Keller admissible triples}

Let $A$ and $B$ be dg algebras, and $X$ be a dg $A$-$B$-bimodule. Let $\rho_A$ and $\rho_B$ be the maps   
\begin{align*}
\rho_A: A \to  \Hom_{\op B}(X,X)  , & \quad \rho_A(a)(x):= a \cdot x, \\
\rho_B: \op B \to  \Hom_A(X,X) , & \quad  \rho_B(b)(x):= (-1)^{|x||b|} x \cdot b.
\end{align*}
The spaces $\Hom_{\op B}(X,X)$ and $\Hom_A(X,X)$ are equipped with the differential $\partial_X = [d_X,\argument]$, where $d_X$ is the given differential on $X$, and $[\argument,\argument]$ is the bracket induced by graded commutators. Note that the maps $\rho_A$ and $\rho_B$ are morphisms of dg algebras.

Since \emph{sum} Hochschild cohomologies are not well-behaved under quasi-isomorphisms, we will require the action maps $\rho_A$ and $\rho_B$ to satisfy a technical condition, called \emph{weak cone-nilpotency} in this paper, so that they induce isomorphisms between Hochschild cohomologies. In order to define weak cone-nilpotency, we first introduce a condition that a sum Hochschild cohomology vanishes.

\subsection{A vanishing condition of Hochschild cohomology}\label{subsection:condition}

Unlike the \emph{product} Hochschild cohomology $\pdHH(A,M)$ of $A$ with values in $M$, which vanishes for an acyclic dg $A$-$A$-bimodule $M$, the \emph{sum} Hochschild cohomology $\grHH(A,M)$ does not necessarily vanish for an acyclic dg $A$-$A$-bimodule $M$. In the following, we describe a condition on an acyclic dg $A$-$A$-bimodule $M$ that the sum Hochschild cohomology $\grHH(A,M)$ vanishes.

Let $(A,d_{A})$ be a dg algebra and $(M,d_{M})$ be an \emph{acyclic} dg $A$-$A$-bimodule (i.e. $H^{\bullet}(M,d_{M})=0$).
Observe that, since the underlying space of $M$ is an acyclic dg vector space over $\KK$, there is a homotopy operator $h:M\to M$ of degree $-1$ such that 
\begin{equation*}
d_{M}\circ h + h \circ d_{M}=\id_{M}.
\end{equation*}
We will refer such a homotopy operator $h:M \to M$ as a \emph{contracting homotopy} for $(M, d_{M})$.

For each fixed $n$, we have the induced map
\[H :=h_{\ast}:\Hom^\bullet (A^{\otimes n}, M)\to \Hom^{\bullet-1} (A^{\otimes n}, M)\] 
of degree $-1$, defined by $H(f)=h\circ f$.

\begin{lemma}\label{lem:AAA}
For each $n$, the induced operator $H$ satisfies the homotopy equation 
$$
\partial\circ H + H\circ \partial = \id_{\Hom(A^{\otimes n}, M)}
$$
on $\big( \Hom^{\bullet}(A^{\otimes n}, M), \partial \big)$, where $\partial$ is the differential induced by the dg structures. 
\end{lemma}
\begin{proof}
Let $f\in \Hom^r (A^{\otimes n}, M)$. 
\begin{align*}
\partial \circ H(f)(a_1\otimes \cdots \otimes a_n) &= d_M \circ h \circ f (a_1\otimes \cdots \otimes a_n) \\
&\quad -(-1)^{r-1} \sum_{i=1}^{n} (-1)^{\degree{a_1}+\cdots + \degree{a_{i-1}}} h\circ f( \cdots \otimes a_{i-1}\otimes d_A(a_i) \otimes a_{i+1}\otimes \cdots )\\
H \circ \partial(f)(a_1\otimes \cdots \otimes a_n) &= h\circ d_M \circ f(a_1\otimes \cdots \otimes a_n)\\
&\quad -(-1)^{r} \sum_{i=1}^{n} (-1)^{\degree{a_1}+\cdots + \degree{a_{i-1}}} h\circ f( \cdots \otimes a_{i-1}\otimes d_A(a_i) \otimes a_{i+1}\otimes \cdots )
\end{align*}
Thus we have $(\partial \circ H+ H \circ \partial)(f)=(d_M\circ h + h\circ d_M)\circ f=f$.
\end{proof}

As a result, we have an operator $H: \grHHcx^{\bullet}(A,M)\to \grHHcx^{\bullet-1}(A,M)$ of degree $-1$ on the Hochschild cochain complex $\big(\grHHcx^\bullet(A,M), \hochschild + \partial \big)$. However, it is \emph{not} a contracting homotopy. Indeed, we have
\begin{align*}
H\circ (\partial + \hochschild) + (\partial + \hochschild)\circ H - \id_{\grHHcx(A,M)} &=H\circ \hochschild + \hochschild \circ H
\end{align*}
which does \emph{not} vanish in general. 

For each $n$, we define a sequence of maps
\begin{equation}\label{eq:frakhk}
 \frakh_{k}:\Hom^{\bullet}(A^{\otimes n}, M) \to \Hom^{\bullet-k-1} (A^{\otimes n+k}, M), \qquad k\geq 0, 
\end{equation}
by 
\[\frakh_{k} :=(H \circ \hochschild)^{k}\circ H = H\circ(\hochschild \circ H)^{k}.\]
In particular, $\frakh_{0}=H=h_{\ast}$. Furthermore, we have 
\[\frakh :=\sum_{k=0}^{\infty}(-1)^{k}\frakh_{k}:\Hom^{\bullet}(A^{\otimes n}, M) \to \prod_{k=0}^{\infty}\Hom^{\bullet-k-1} (A^{\otimes n+k}, M),\] 
which defines a map $\frakh:\grHHcx^{\bullet}(A,M) \to \pdHHcx^{\bullet-1}(A,M)$ of degree $-1$. 
Note that this map $\frakh$ depends on a choice of contracting homotopy $h$ for $M$.

\begin{lemma}
The map $\frakh:\grHHcx^{\bullet}(A,M) \to \pdHHcx^{\bullet-1}(A,M)$ satisfies the equation 
$$
\frakh \circ (\partial + \hochschild)(f) + (\partial + \hochschild)\circ \frakh(f) = f 
$$
for $f \in \grHHcx^{\bullet}(A,M)$.
\end{lemma}
\begin{proof}
It suffices to show that $\partial  \frakh_{k} + \frakh_{k} \partial=\hochschild  \frakh_{k-1} + \frakh_{k-1}  \hochschild$. 
Recall that we have
\begin{gather*}
\partial H + H \partial = \id\\
\partial  \hochschild + \hochschild  \partial = 0\\
\hochschild^2 = 0.
\end{gather*}
We prove the assertion by induction on $k$. For $k=1$, we consider $\frakh_0=H$ and $\frakh_1=H\hochschild H$. Since
\[\partial (H\hochschild H)= (\id - H\partial) \hochschild H = \hochschild H + H \hochschild \partial H = \hochschild H + H\hochschild (\id-H\partial),\] 
we have 
$$
\partial  \frakh_{1} + \frakh_{1} \partial=\hochschild  \frakh_0 + \frakh_0 \hochschild.
$$

Suppose that the equation $ \partial \frakh_{n} + \frakh_{n}  \partial=\hochschild  \frakh_{n-1} + \frakh_{n-1} \hochschild  $ holds for $n= k$. Then we have 
\begin{align*}
\partial \frakh_{k+1} &= \partial H \hochschild \frakh_k = (\id- H\partial) \hochschild \frakh_k =\hochschild \frakh_k + H\hochschild \partial \frakh_k \\
&= \hochschild \frakh_k + H\hochschild (\hochschild \frakh_{k-1}+\frakh_{k-1}\hochschild - \frakh_k \partial) =\hochschild \frakh_k +\frakh_{k}\hochschild - \frakh_{k+1} \partial .
\end{align*}
This proves the lemma.
\end{proof}

As a result, the map $\frakh:\grHHcx^{\bullet}(A,M) \to \pdHHcx^{\bullet-1}(A,M)$ defines a contracting homotopy for the sum Hochschild complex $\big(\grHHcx^\bullet(A,M), \hochschild +\partial \big)$ if and only if for each $f\in \grHHcx^{\bullet}(A,M)$, there exists an integer $N=N(f)$ such that $\frakh_{k}(f)=0$ if $k>N$. 

\begin{definition}
We say an acyclic dg $A$-$A$-bimodule $M$ is \textbf{pointwisely nilpotent} if there exists a contracting operator $h$ for $(M,d_{M})$ such that the induced sequence $\{ \frakh_{k}  \}_{k=0}^\infty$ of maps defined in \eqref{eq:frakhk} satisfies the following: for each $f\in \grHHcx^{\bullet}(A,M)$, there exists an integer $N=N(f)$ such that $\frakh_{k}(f)=0$ for all $k>N$. 
\end{definition}

\begin{proposition}\label{prop:PtwiseNilpotent}
Let $A$ be a dg algebra, and $M$ be an acyclic dg $A$-$A$-bimodule. Then $\grHH^\bullet (A,M)=0$ if $M$ is pointwisely nilpotent.
\end{proposition}

If the dg algebra $A$ is finite dimensional, the pointwise nilpotency of an acyclic dg $A$-$A$-bimodule $M$ is obtained by an ascending filtration of $M$ compatible with a contracting homotopy for $M$.
\begin{lemma}\label{lem:Mk}
Let $A$ be a finite dimensional dg algebra. An acyclic dg $A$-$A$-bimodule $M$ is pointwisely nilpotent if there exists an ascending filtration 
\[\{0\}=M_{0}\subset M_{1}\subset \cdots \subset M\]
of $A$-$A$-subbimodules of $M$ and a contracting homotopy $h$ for $M$ such that $\bigcup_{q\geq 0}M_{q}=M$ and 
\[h(M_{q+1})\subset M_{q}\]
for all $q\geq 0$.
\end{lemma}

Note that each $M_{i}$ is an $A$-$A$-subbimodule of $M$, but not necessarily a \textit{dg} $A$-$A$-subbimodule of $M$.

\begin{proof}
Let $f\in \Hom^{r}(A^{\otimes p},M)$. Observe that if $f(A^{\otimes p})\subset M_{q}$, then $(\hochschild \circ H)(f) \in \Hom^{r-1}(A^{\otimes p+1},M)$ and 
\begin{equation}\label{eq:lem:Mk}
(\hochschild \circ H)(f)(A^{\otimes p+1})\subset A\cdot h(M_{q})+h(M_{q})\cdot A \subset M_{q-1}.
\end{equation}
Since $A$ is finite dimensional, there is $k\geq 0$ such that $f(A^{\otimes p})\subset M_{k}$. 
By applying \eqref{eq:lem:Mk} repeatedly, we have 
\[\frakh_{k}(f)(A^{\otimes p+k})=H\circ (\hochschild \circ H)^{k}(f)(A^{\otimes p+k}) \subset h(M_{0}) =\{0\}\]
which proves the lemma.
\end{proof}

\subsection{Weak cone-nilpotency and Keller admissible triples}\label{sec:WCN}

We introduce a technical condition on quasi-isomorphisms, called weakly cone-nilpotency, to define Keller admissible triples. 

\begin{definition}
Let $A$ be a dg algebra, and let $M$ and $N$ be dg $A$-$A$-bimodules. 
A quasi-isomorphism $\phi: M \to N$  of $A$-$A$-bimodules is said to be \textbf{weakly cone-nilpotent} if one of the following conditions is satisfied: 
\begin{itemize}
\item[(i)]
The mapping cone of $\phi$ is pointwisely nilpotent.
\item[(ii)]
The map $\phi$ has a right inverse whose mapping cone is pointwisely nilpotent.
\item[(iii)]
The map $\phi$ has a left inverse whose mapping cone is pointwisely nilpotent.
\end{itemize}
\end{definition}

\begin{remark}
The above conditions are not equivalent to each other. In fact, the weakly cone-nilpotent quasi-isomorphisms $\rho_A$ and $\rho_B$ in Section~\ref{sec:ActionMap} do not have right inverses. The map $\rho_A$ satisfies (i), and the map $\rho_B$ satisfies (iii). None of them satisfies (ii).
%
\end{remark}

Let $\phi:\cC^{\bullet} \to \cD^{\bullet}$ be a map of cochain complexes $(\cC^{\bullet},d_{\cC})$ and $(\cD^{\bullet},d_{\cD})$. In the present paper, the mapping cone of $\phi$ is a cochain complex $\Cone^{\bullet}(\phi)=(\cC^{\bullet+1}\oplus \cD^{\bullet},d_{\Cone(\phi)})$ where the differential is defined by
\[d_{\Cone(\phi)}(c,d)=\big(-d_{\cC}(c),\phi(c)+d_{\cD}(d)\big)\]
for $c\in \cC$ and $d \in \cD$. Note that if $(\cC,d_{\cC})$ and $(\cD,d_{\cD})$ are dg $A$-$A$-bimodules, then $\Cone(\phi)$ also carries the dg $A$-$A$-bimodule structures.

\begin{lemma}\label{lem:SumHochIso}
Let $\phi : M \to N$ be a quasi-isomorphism of $A$-$A$-bimodules. If $\phi$ is weakly cone-nilpotent, then the induced map
\[\phi_{\ast}:\grHHcx(A,M)\to \grHHcx(A,N)\]
is a quasi-isomorphism.
\end{lemma}
\begin{proof}
If the mapping cone $\Cone(\phi)$ of $\phi$ is pointwisely nilpotent, then the lemma follows from Proposition~\ref{prop:PtwiseNilpotent} and the fact that 
\[\Cone(\phi_{\ast})\cong \grHHcx(A,\Cone(\phi))\]
where $\Cone(\phi_{\ast})$ is the mapping cone of the pushforward map $\phi_{\ast}:\grHHcx(A,M)\to \grHHcx(A,N)$. 

Let $\tau$ be a right/left inverse of $\phi$. By the same reason, if the mapping cone of $\tau$ is pointwisely nilpotent, then the pushforward map $\tau_\ast:\grHHcx(A,N)\to \grHHcx(A,M)$ is a quasi-isomorphism. Since $\tau_\ast$ is a right/left inverse of $\phi_\ast$, the map $\phi_\ast$ is also a quasi-isomorphism. 
\end{proof}

\begin{definition}
A triple $(A,X,B)$ is called a \textbf{Keller admissible triple} if (i) the action maps 
\begin{gather*}
\rho_A: (A,d_A) \to  \big(\Hom_{\op B}(X,X), \partial_X\big) , \\
\rho_B: (\op B,d_B) \to  \big(\Hom_A(X,X), \partial_X \big)
\end{gather*}
are weakly cone-nilpotent quasi-isomorphisms, and (ii) the sequences 
\begin{gather}
0 \to \Hom_{\op B}(X,X) \into \Hom(X,X) \xto{{\hochschild^X}_R} \Hom(X \otimes B, X) \xto{{\hochschild^X}_R} \Hom(X \otimes B^{\otimes 2}, X) \xto{{\hochschild^X}_R} \cdots  \label{eq:KellerSeqA} \\
0 \to \Hom_{A}(X,X) \into \Hom(X,X) \xto{{\hochschild^X}_L} \Hom(A \otimes X, X) \xto{{\hochschild^X}_L} \Hom(A^{\otimes 2} \otimes X, X) \xto{{\hochschild^X}_L} \cdots \label{eq:KellerSeqB}
\end{gather}
are exact.
\end{definition}

The main theorem of this paper is the following

\begin{theorem}\label{thm:d.s.KellerThm}
Let $(A,X,B)$ be a Keller admissible triple. The two projections 
$$
\begin{tikzcd}
 & \ar[ld,two heads,"\pi_A"'] \grHHcx^\bullet(X) 
  \ar[rd,two heads,"\pi_B"] & \\
 \grHHcx^\bullet(A) & & \grHHcx^\bullet(B)
\end{tikzcd}
$$
induce isomorphisms of Gerstenhaber algebras on cohomologies. 
\end{theorem}

\begin{remark}
A similar theorem for \textit{product} Hochschild cohomology appears in \cite{keller2003derived}. However, there is \textit{no} obvious relations between the two versions. See Example~\ref{ex:KellerAdmissibleSumHoch} below.
\end{remark}

The rest of the section is devoted to proving Theorem~\ref{thm:d.s.KellerThm}.

\subsection{Hochschild complexes with values in $\Hom_{\op B}(X,X)$ and $\Hom_{A}(X,X)$}

Since $\Hom_{\op B}(X,X)$ is equipped with the $A$-$A$-bimodule structure 
$$
(a\cdot f)(x):= a \cdot f(x), \qquad (f\cdot a)(x):= f(a \cdot x),
$$ 
we have the Hochschild complex $\grHHcx^\bullet\big(A,\Hom_{\op B}(X,X) \big)$ of $A$ with values in $\Hom_{\op B}(X,X)$. Similarly, with the $B$-$B$-bimodule structure 
$$
(b\cdot f)(x):= (-1)^{|b|(|f|+|x|)}f(x \cdot b), \qquad (f \cdot b)(x):= (-1)^{|b||x|}f(x)\cdot b,
$$
on $\Hom_{A}(X,X)$, we have the Hochschild complex $\grHHcx^\bullet\big(B,\Hom_{A}(X,X) \big)$ with values in $\Hom_{A}(X,X)$.

Suppose $X$ is a dg $A$-$B$-bimodule such that  the sequence \eqref{eq:KellerSeqA} is exact. For each $p,r$, since the functor $\Hom_\KK (A^{\otimes p},-)$ is exact, we have an exact sequence of vector spaces:
\begin{equation}\label{eq:KellerSeqA1}
\begin{aligned}
0\rightarrow \Hom^{r} \big(A^{\otimes p}, & \Hom_{\op B}(X,X)\big) \xhookrightarrow{} \Hom^{r} \big(A^{\otimes p} , \Hom(X , X )\big)\to \\
& \to \Hom^{r} \big(A^{\otimes p} , \Hom(X \otimes B, X )\big) \to \Hom^{r} \big(A^{\otimes p} , \Hom(X \otimes B^{\otimes 2}, X )\big)\to \cdots
\end{aligned}
\end{equation}
Under the isomorphism
\begin{align*}
\Hom^{r}(A^{\otimes p} \otimes X \otimes B^{\otimes q}, X) \cong \Hom^{r}\big(A^{\otimes p} , \Hom(X \otimes B^{\otimes q}, X )\big),
\end{align*}
the sequence \eqref{eq:KellerSeqA1} can be rephrased as the exact sequence
\begin{equation*}
0\to \Hom^r\big(A^{\otimes p}, \Hom_{\op B}(X,X)\big) \xhookrightarrow{\Phi^{p,r}} \HHcx^{p,0,r}(A,X,B) \xto{{\hochschild^X}_R} \HHcx^{p,1,r}(A,X,B) \xto{{\hochschild^X}_R} \cdots
\end{equation*}
for each $p,r$.

We define $\Phi=\sum_{p,r}(-1)^{r}\ \Phi^{p,r}$. Then we have a double complex 
\begin{equation}\label{doublecomplex}
\begin{tikzcd}
&\vdots&\vdots&\vdots&\\
0 \arrow[r]&\mathcal{C}^{-1,n+1} \arrow[r,"\Phi"]\arrow[u] & \mathcal{C}^{0,n+1} \arrow[r,"{\hochschild^X}_R"]\arrow[u]&\mathcal{C}^{1,n+1} \arrow[r, "{\hochschild^X}_R"]\arrow[u]&\cdots\\
0 \arrow[r] & \mathcal{C}^{-1,n} \arrow[r, " \Phi"] \arrow[u, "\hochschild + \partial"] & \mathcal{C}^{0,n} \arrow[r, "{\hochschild^X}_R"]\arrow[u, "{\hochschild^{X}}_{L} + \partial_{X}"']&\mathcal{C}^{1,n} \arrow[r, "{\hochschild^X}_R"]\arrow[u, "{\hochschild^{X}}_{L} + \partial_{X}"']&\cdots\\
&\vdots\arrow[u]&\vdots\arrow[u]&\vdots\arrow[u]&
\end{tikzcd}
\end{equation}
where 
\begin{gather*}
\mathcal{C}^{-1,n}=\bigoplus_{p+r=n-1} \Hom^r\big(A^{\otimes p}, \Hom_{\op B}(X,X)\big),\\
\mathcal{C}^{q,n}= \bigoplus_{p+r+1=n}\HHcx^{p,q,r}(A,X,B), \quad q\geq 0.
\end{gather*}
Here, the complex $\big( \mathcal{C}^{-1,\bullet+1}, \ \hochschild + \partial \big)$ is the Hochschild cochain complex of $A$ with values in $\Hom_{\op B}(X,X)$, and the total complex $\big( \bigoplus_{\substack{q \geq 0, \\ q+n= \bullet}}\mathcal{C}^{q,n}, \ {\hochschild^{X}}_{L} + \partial_{X} +{\hochschild^{X}}_{R}\big)$ is a subcomplex of $\big( \grHHcx^{\bullet}(X), \ \hochschild^{X} + \partial_{X}\big)$.

Similarly, one has the map $\Psi: \grHHcx^{\bullet} (B, \Hom_{A}(X,X)) \to \grHHcx^{\bullet+1}(A,X,B),$
$$
\Psi(f)(x \, ;b_1, \cdots, b_q):= (-1)^{q+r-1}(-1)^{|x|(|b_1|+\cdots+|b_q|)}f(b_1 \otimes \cdots \otimes b_q)(x),
$$
for $f \in \Hom^r(B^{\otimes q}, \Hom_{A}(X,X))$ and has a double complex analogous to \eqref{doublecomplex}.

We need the following lemma in Homological algebra \cite[Lemma 2.7.3]{MR1269324}:
\begin{lemma}[Acyclic Assembly Lemma] \label{lem:AAL}
Let $\mathcal{B}^{p,q}$ be a double complex, equipped with differentials 
\begin{gather*}
d_{1}:\mathcal{B}^{p,q}\to \mathcal{B}^{p+1,q}\\
d_{2}:\mathcal{B}^{p,q} \to \mathcal{B}^{p,q+1}.
\end{gather*}
 Suppose that the 2nd quadrant of $\mathcal{B}$ vanishes (i.e. $\mathcal{B}^{p,q}=0$ if $p<0$ and $q>0$). If every row of $\mathcal{B}$ is exact, then the direct sum total complex $\tot^{\oplus}(\mathcal{B})$ is acyclic.
\end{lemma}

Applying Lemma~\ref{lem:AAL} to \eqref{doublecomplex}, we get

\begin{lemma}\label{prop:EmbedIsQisoKAT}
Let $A$ and $B$ be dg algebras, and $X$ be a dg $A$-$B$-bimodule. If the sequence \eqref{eq:KellerSeqA} is exact, then the map
\[\Phi: \big( \grHHcx^{\bullet}(A, \Hom_{\op B}(X,X) )[-1], \  -(\hochschild+\partial) \big) \xto{}
 \big( \grHHcx^{\bullet}(A,X,B), \ \hochschild^{X}+ \partial_{X}\big),\]
is a quasi-isomorphism.

Similarly, if the sequence \eqref{eq:KellerSeqB} is exact, then the map
\[\Psi:\big( \grHHcx^{\bullet}(B, \Hom_{A}(X,X) )[-1], \  -(\hochschild+\partial) \big) \xto{}
 \big( \grHHcx^{\bullet}(A,X,B), \ \hochschild^{X}+ \partial_{X}\big),\]
is a quasi-isomorphism.
\end{lemma}
\begin{proof}
We show the assertion for the first part, and the second part follows from a similar argument. 

It is clear that the map $\Phi$ is a cochain map. Observe that the mapping cone $\cone{\Phi}$ of $\Phi$ is the direct sum total complex $\tot^{\oplus}(\mathcal{C})$ of \eqref{doublecomplex}. By Lemma~\ref{lem:AAL}, the total complex $\tot^{\oplus}(\mathcal{C})=\cone{\Phi}$ is acyclic, and therefore, we conclude that $\Phi$ is a quasi-isomorphism.
\end{proof}

\subsection{Proof of Theorem~\ref{thm:d.s.KellerThm}}

Let $(A,X,B)$ be a Keller admissible triple. Observe that the action map  $\rho_A: (A,d_A) \to \big( \Hom_{\op B}(X,X), \partial_X \big)$ is a dg $A$-$A$-bimodule map, and $\rho_B: (\op B,d_B) \to \big( \Hom_{A}(X,X), \partial_X \big)$ is a dg $B$-$B$-bimodule map. Thus, by Lemma~\ref{lem:SumHochIso}, the induced maps
\begin{align*}
{\rho_A}_\ast: \Big( \grHHcx^\bullet(A), \ -(\hochschild^A + \partial_A) \Big) & \to  \Big( \grHHcx^\bullet\big(A, \Hom_{\op B}(X,X)\big), \ -(\hochschild + \partial) \Big),  \\
{\rho_B}_\ast: \Big( \grHHcx^\bullet(B), \ -(\hochschild^B + \partial_B) \Big) & \to  \Big( \grHHcx^\bullet\big(B, \Hom_{A}(X,X)\big), \ -(\hochschild + \partial) \Big)
\end{align*}
are quasi-isomorphisms.

To prove Theorem~\ref{thm:d.s.KellerThm}, we need the following

\begin{lemma}\label{lem:ProjQiso}
If the map 
$$
\Phi \circ {\rho_A}_\ast: \big(\grHHcx^\bullet(A)[-1], \ -(\hochschild^A + \partial_A) \ \big) \to \big( \grHHcx^{\bullet}(A,X,B), \ \hochschild^X + \partial_X \ \big)
$$ 
is a quasi-isomorphism of cochain complexes, then so is the projection $\pi_B: \grHHcx^\bullet(X) \onto \grHHcx^\bullet(B)$. 
Similarly, if 
$$\Psi \circ {\rho_B}_\ast:\big(\grHHcx^\bullet(B)[-1], -(\hochschild^B + \partial_B)  \big) \to \big( \grHHcx^{\bullet}(A,X,B),  \hochschild^X + \partial_X  \big)$$ 
is a quasi-isomorphism, then so is the projection $\pi_A: \grHHcx^\bullet(X) \onto \grHHcx^\bullet(A)$. 
\end{lemma}
\begin{proof}
We only prove the first statement. The second one follows from a similar argument.
 
Since the projection $\pi_B: \big(\grHHcx^\bullet(X), \hochschild + \partial \big) \onto \big(\grHHcx^\bullet(B), \hochschild^B + \partial_B \big)$ is surjective, it is a quasi-isomorphism if and only if the kernel 
$$
\ker(\pi_B) = \big( \grHHcx^\bullet(A) \oplus \grHHcx^{\bullet }(A,X,B)  , \ \hochschild^{AX}  + \hochschild^A + \hochschild^X +  \partial_X + \partial_A \ \big)
$$
is acyclic. Since  
\begin{align*}
 \Phi \circ {\rho_A}_\ast(f)(a_1,\cdots, a_p; \, x) & = (-1)^{r} f(a_1,\cdots, a_p)\cdot x \\
& = \hochschild^{AX}(f)(a_1,\cdots, a_p; \, x),
\end{align*}
for $f \in \HHcx^{p,r}(A)$, we have $\Phi \circ {\rho_A}_\ast=\hochschild^{AX}$. Thus the mapping cone 
$$
 \Cone(\Phi \circ {\rho_A}_\ast) = \big( \grHHcx^\bullet(A) \oplus \grHHcx^{\bullet }(A,X,B)  , \  \Phi \circ {\rho_A}_\ast  + \hochschild^A + \hochschild^X +  \partial_X + \partial_A \ \big)
$$
of $\Phi \circ {\rho_A}_\ast$ coincides with the kernel $ \ker(\pi_B)$. 
Since $\Phi \circ {\rho_A}_\ast$ is a quasi-isomorphism, the mapping cone $\Cone(\Phi \circ {\rho_A}_\ast) = \ker(\pi_B)$ is acyclic. Therefore, $\pi_B$ is a quasi-isomorphism. 
\end{proof}

\begin{pfdsKellerThm}
Let $(A,X,B)$ be a Keller admissible triple. By Lemma~\ref{prop:EmbedIsQisoKAT}, the embedding maps $\Phi$ and $\Psi$ are quasi-isomorphisms. Since the pushforward maps ${\rho_A}_\ast$ and ${\rho_B}_\ast$ are also quasi-isomorphisms, it follows from Lemma~\ref{lem:ProjQiso} that $\pi_A$ and $\pi_B$ are  quasi-isomorphisms. 
\end{pfdsKellerThm}

\begin{remark}
Theorem~\ref{thm:d.s.KellerThm} is a sum analogue of Keller's theorem in \cite{keller2003derived} which respects the \emph{product} Hochschild cohomology.  
In fact, in \cite{keller2003derived,MR2628794,MR3941473}, analogous theorems were proved for bigraded algebras where a Hochschild complex is endowed with three types of degrees --- one from Hochschild construction, two from the given bigrading. Although a direct sum is taken for bigraded components in \cite{MR2628794}, the total complex is still different from our sum Hochschild complexes. Thus, there is no clear relation between these theorems and Theorem~\ref{thm:d.s.KellerThm}.
\end{remark}

\section{Keller admissible triples associated with Lie algebras}\label{sec:KellerTrip-LieAlg}

Let $\frakg$ be a finite-dimensional Lie algebra. 
In this section, we prove that the triple 
\begin{align*}
(A,d_A) & =  \big(\cU\frakg, 0\big),\\
(B,d_B) & = \big(\Hom(S(\frakg[1]),\KK) , d_{\frakg}\big) \cong \big(S(\frakg[1])\dual, d_{\frakg}\big) ,
\\
(X,d_X) & = (\cU\frakg  \otimes S(\frakg[1]),d_X)  
\end{align*} 
is a Keller admissible triple, 
where $d_\frakg: S^\bullet(\frakg[1])\dual \to S^{\bullet +1}(\frakg[1])\dual $ is the Chevalley--Eilenberg differential defined as in \eqref{eq:LieAlgDiff},  $\cU\frakg$ is the universal enveloping algebra of $\frakg$, and  $d_X :\cU\frakg  \otimes S(\frakg[1])  \to \cU\frakg  \otimes S(\frakg[1]) $ is defined by
\begin{equation*}
\begin{split}
d_X(u \otimes x_1 \odot \cdots \odot x_n ) & := \sum_{i=1}^n (-1)^{i+1} \  u \cdot \suspend x_i \otimes x_1 \odot \cdots \widehat{x_i} \cdots \odot x_n  \\
 & \qquad + \sum_{i<j} (-1)^{i+j} \ u \otimes \suspend\inv [\suspend x_i, \suspend x_j]_\frakg \odot x_1 \cdots \widehat{x_i} \cdots \widehat{x_j} \cdots \odot x_n .
\end{split}
\end{equation*}
Here, $\cU\frakg$ is considered as a dg algebra concentrated at degree zero, $\suspend: \frakg[1] \to \frakg$ is the degree-shifting map of degree $+1$, $u \in \cU\frakg$ and $x_1, \cdots, x_n \in \frakg[1]$. This triple is adapted from \cite[Example~6.5]{MR1258406}.

It is well known that the complex $\big(\cU\frakg  \otimes S(\frakg[1]),   d_X \big)$ is a free resolution of $\KK$ in the category of $\cU\frakg$-modules.

\begin{remark}
In our grading setting, the degrees are chosen to be compatible with the Koszul sign convention, and a few classical formulations need to be modified correspondingly.  
In fact, in the literature, the expressions $(\Lambda^\bullet \frakg\dual, d_\frakg)$ and $(\cU\frakg  \otimes \Lambda^\bullet\frakg, \  d_X )$ are more common than the expressions of graded symmetric tensors in this paper. The two types of expressions are isomorphic as complexes. Nevertheless, in the category of graded vector spaces, an element in $\Lambda^\bullet\frakg$ should be considered to be of degree zero which is \emph{not} the expected degree. Thus, in order to avoid confusion and to keep the consistency of degree counting, we choose the expressions $S(\frakg[1])\dual$ and $\cU\frakg  \otimes S(\frakg[1])$. See Appendix~\ref{sec:ConvAlg} for more details.  
\end{remark}

We describe the right $S(\frakg[1])\dual$-action $\Sgvaction$ on $S(\frakg[1])$ induced by (graded) contraction: For any $x, x_i\in \frakg[1]$, $\xi \in (\frakg[1])\dual$, we define   
\begin{gather*}
1 \Sgvaction \xi  :=0, \\
(x_1 \odot \cdots \odot x_n) {\Sgvaction} \xi := \sum_{i=1}^n (-1)^{n-i} \pair{x_i}{\xi} (x_1 \odot \cdots \widehat{x_i} \cdots \odot x_n), 
\end{gather*}
and extend it by the module axiom. 
Similarly, we also have the right $S(\frakg[1])$-action $\Sgaction$ on $S(\frakg[1])\dual$. 
Note that, in this way, we have
\begin{gather*}
x \Sgvaction \xi = \pair{x}{\xi} = -\xi(x),\\
\xi \Sgaction x = \pair{\xi}{x} = \xi(x).
\end{gather*}

The space $X = \cU \frakg  \otimes S(\frakg[1])$ carries the left $\cU\frakg$-action induced by the multiplication in $\cU\frakg$ and the right $S(\frakg[1])\dual$-action induced by $\Sgvaction$.

\begin{lemma}
The dg vector space $\big(\cU\frakg  \otimes S(\frakg[1]), \  d_X \big)$ is a dg $\cU\frakg$--$S(\frakg[1])\dual$-bimodule.
\end{lemma}
\begin{proof}
This lemma follows from Example~\ref{ex:SgcoMod}.
\end{proof}

\subsection{The action maps}\label{sec:ActionMap}

We first prove the action maps are weakly cone-nilpotent quasi-isomorphisms.

\begin{lemma}\label{lem:action1}
The action map
\begin{align*}
& \rho_A: (\cU\frakg, 0 )  \to \big( \Hom_{\op{(S( \frakg[1])\dual)}} \big(\cU\frakg \otimes S( \frakg[1]), \cU\frakg \otimes S( \frakg[1])\big), \partial_X \big) , \\
& \qquad  \rho_A(v)(u \otimes {\bf x})  = (vu) \otimes {\bf x},
\end{align*}
is a weakly cone-nilpotent quasi-isomorphism. 
\end{lemma}
\begin{proof}
Since $B=S( \frakg[1])\dual$ is a graded commutative algebra, a right $S( \frakg[1])\dual$-module structure is equivalent to a left $S( \frakg[1])\dual$-module structure. Thus, by Proposition~\ref{prop:coModv.s.Mod} and Example~\ref{ex:SgcoMod}, we have 
$$
\Hom_{\op{(S( \frakg[1])\dual)}}^\bullet(X,X) = \coHom_{S( \frakg[1])}^\bullet(X,X),
$$
where $X = \cU \frakg \otimes S( \frakg[1])$. 
Furthermore, since the natural projection
$$
\pr: \cU \frakg \otimes S( \frakg[1]) \onto  \cU\frakg \otimes S^0( \frakg[1]) \cong \cU\frakg
$$
is a free cogenerator of the graded comodule $X= \cU \frakg \otimes S( \frakg[1])$, it follows from Proposition~\ref{prop:coGenerator} that the induced map
$$
\pr_\ast : \Hom_{\op B}^\bullet(X,X) \to \Hom_{\KK}^\bullet(\cU \frakg \otimes S( \frakg[1]), \cU\frakg) \cong \Hom_{\KK}^\bullet\big(S(\frakg[1]),\Hom_\KK(\cU\frakg,\cU\frakg)\big)
$$
is an isomorphism of graded vector spaces. See Appendix~\ref{sec:DGcoAlg} for details. 

Let $\tilde\partial_X$ be the differential on $\Hom\big(S(\frakg[1]),\Hom(\cU\frakg,\cU\frakg)\big)$ induced by $\partial_X$ under the isomorphism $\pr_\ast$. By \eqref{eq:coGenModMor}, one can show that 
\begin{align*}
\tilde\partial_X(f)( x_0 \odot \cdots \odot x_n)(u)&  = \sum_{i=0}^n (-1)^{n-i} \big(f( x_0 \odot \cdots \widehat{x_i} \cdots \odot x_n)(u)  \suspend x_i -f( x_0 \odot \cdots \widehat{x_i} \cdots \odot x_n)(u \suspend x_i)\big) \\
& \qquad -(-1)^n \sum_{i<j} (-1)^{i+j} f(\suspend\inv [\suspend x_i, \suspend x_j]_\frakg \odot x_0 \cdots \widehat{x_i} \cdots \widehat{x_j} \cdots \odot x_n)(u)  
\end{align*}
for $f \in \Hom\big(S^n(\frakg[1]),\Hom(\cU\frakg,\cU\frakg)\big)$, $x_i \in \frakg[1]$ and $u \in \cU\frakg$. Let $\action : \frakg \times \Hom(\cU\frakg,\cU\frakg) \to \Hom(\cU\frakg,\cU\frakg)$ be the Lie algebra action
$$
(y\action g)(u) = g(uy) - g(u) y 
$$
for $g \in \Hom(\cU\frakg,\cU\frakg)$ and $y \in \frakg$. Then the differential $\tilde\partial_X$ coincides with the Chevalley--Eilenberg differential $d_{\CE}^\action$ associated with the action $\action$.  

Recall that one has the isomorphism (see, for example, \cite[Exercise~7.3.5]{MR1269324}) 
$$
H_{\CE}^\bullet\big(\frakg, \Hom(\cU\frakg, \cU\frakg) \big) \cong \Ext_{\op{\cU\frakg}}^\bullet(\cU\frakg, \cU\frakg) \cong \cU\frakg.
$$

Since the action map $\rho_A$ induces a right inverse of the isomorphism $\pr_\ast:H\big(\Hom_{\op B}^\bullet(X,X), \partial_X\big) \to \cU\frakg$, the lemma follows.

For weak cone-nilpotency, we observe that the dg algebra $\cU\frakg$ is concentrated in degree $0$ and the mapping cone $M:=\Cone(\rho_{A})$ of $\rho_{A}$ is bounded below. Thus, there exists $N>0$ such that $\Hom^{r}(A^{\otimes n}, M)=0$ if $r<-N$. Therefore, for any choice of homotopy operator on $M$ and for any $f\in \grHHcx^{s}(A,M)$, the induced maps $\{\frakh_{k}\}$ satisfy $\frakh_{k}(f)=0$ for $k>s+N$. This implies that the mapping cone $M=\Cone(\rho_{A})$ is pointwisely nilpotent.
\end{proof}

The quasi-isomorphism property for the other action 
\[\rho_{B}:\big( \op{\big(S(\frakg[1])\dual\big)} ,  d_\frakg \big)  \to \big( \Hom_{\cU\frakg} \big(\cU\frakg  \otimes S(\frakg[1]), \cU\frakg  \otimes S(\frakg[1])\big), \partial_X \big)\]
follows from the proof of \cite[Lemma~6.5~(a)]{MR1258406}. In fact, the augmentation map $\varepsilon:\cU\frakg  \otimes S(\frakg[1]) \onto \cU\frakg \onto \KK$ induces a quasi-isomorphism 
$$
\varepsilon_\ast : \big( \Hom_{\cU\frakg}(\cU\frakg  \otimes S(\frakg[1]), \cU\frakg  \otimes S(\frakg[1])), \partial_X \big) \to  \big( \Hom_{\cU\frakg}\big(\cU\frakg  \otimes S(\frakg[1]), \KK \big),  \partial_{X, \KK} \big) \cong \big( S(\frakg[1])\dual ,  d_\frakg \big)
$$ 
which defines a left inverse for $\rho_B$. 

\begin{lemma}\label{lem:action2}
The mapping cone of 
$$
\varepsilon_\ast : \big( \Hom_{\cU\frakg}(\cU\frakg  \otimes S(\frakg[1]), \cU\frakg  \otimes S(\frakg[1])), \partial_X \big) \to   \big( \op{(S(\frakg[1])\dual)} ,  d_\frakg \big)
$$ 
is pointwisely nilpotent. In particular, the action map
\begin{align*}
& \rho_B: \big( \op{\big(S(\frakg[1])\dual\big)} ,  d_\frakg \big)  \to \big( \Hom_{\cU\frakg} \big(\cU\frakg  \otimes S(\frakg[1]), \cU\frakg  \otimes S(\frakg[1])\big), \partial_X \big), \\
& \qquad \rho_B(b)(u \otimes {\bf x})  = (-1)^{|{\bf x}||b|} u \otimes ({\bf x} \Sgvaction b), 
\end{align*}
is a weakly cone-nilpotent quasi-isomorphism.
\end{lemma}
\begin{proof}
The mapping cone of $\varepsilon_\ast$ is 
\[\Cone(\varepsilon_\ast)\cong \Hom_{\cU\frakg}\Big(\cU\frakg \otimes S(\frakg[1]), \big(\cU \frakg \otimes S(\frakg[1])\big)[1]\oplus \KK\Big)\]
equipped with the differential $\partial:\Cone(\varepsilon_\ast) \to \Cone(\varepsilon_\ast)$,
\[
\partial(f):=d_{\varepsilon}\circ f -(-1)^{\degree{f}}f\circ d_{X}\]
where $d_{\varepsilon}$ is the differential on the mapping cone $\Cone(\varepsilon)$ of $\varepsilon: \cU\frakg \otimes S(\frakg[1])\to \KK$. 

Let $\{e_1, \cdots, e_d\}$ be a basis for $\frakg[1]$, $\suspend:\frakg[1] \to \frakg$ be the degree-shifting map of degree $+1$, and $\dsuspend: \cU\frakg \otimes S(\frakg[1]) \to \big(\cU\frakg \otimes S(\frakg[1])\big)[1]$ be the degree-shifting map of degree $-1$. Note that the mapping cone $\Cone(\varepsilon)$ is a filtered complex with respect to the filtration 
$$
F^{-p} := \Big(\bigoplus_{k=0}^p \cU\frakg^{\leq p-k} \otimes S^{k}(\frakg[1])\Big)[1] \oplus \KK
$$
whose basis is the set
$$
\{(0,1)\} \cup \{(\dsuspend (\suspend e_{i_{1}}\cdots \suspend e_{i_{l}}\otimes e_{j_{1}}\odot \cdots \odot e_{j_{k}}), 0) \mid  i_{1}\leq \cdots \leq i_{l}, \  j_{1} < \cdots < j_{k}, \  k+l \leq p  \}
$$
for $p\geq 0$, and $F^{-p}=0$ if $p<0$.
It is known \cite[Lemma~VII.4.1]{MR1438546} that the quotient complex $F^{-p}/F^{-p+1}$ is exact. As a result, the inclusion map $F^{-p+1} \into F^{-p}$ is a quasi-isomorphism, and thus $F^{-p}$ is exact for each $p$. 
See \cite[Section VII.4]{MR1438546}  or \cite[Section XIII.7]{MR0077480} for details.

We will inductively construct a sequence of homotopy operators $h_p:F^{-p} \to F^{-p}$ of degree $-1$ such that (i) $h_p|_{F^{-p+1}} = h_{p-1}$, and (ii) the equation $d_\varepsilon \circ h_p + h_p \circ d_\varepsilon =\id_{F^{-p}}$ holds in $F^{-p}$. Since $\cup_p F^{-p} = \Cone(\varepsilon)$, such a sequence defines a contracting homotopy $h:\Cone(\varepsilon) \to \Cone(\varepsilon)$ for $\Cone(\varepsilon)$ with the property $h(F^{-p}) \subset F^{-p}$.

In the case $p=0$, one can choose $h_0$ to be the isomorphism $\KK \xto\cong \cU\frakg^{\leq 0}  \otimes S^0(\frakg[1])$. Assume we have a homotopy operator $h_p:F^{-p} \to F^{-p}$ with properties (i) and (ii). Then since $F^{-p-1}$ is an exact sequence of vector spaces, one can assign the value of $h_{p+1}$ at 
$$
(\dsuspend (\suspend e_{i_{1}}\cdots \suspend e_{i_{l}}\otimes e_{j_{1}}\odot \cdots \odot e_{j_{k}}), 0), \quad i_{1}\leq \cdots \leq i_{l}, \  j_{1} < \cdots < j_{k}, \  k+l =p+1,
$$
inductively on $k$ so that the equation $d_\varepsilon \circ h_{p+1} + h_{p+1} \circ d_\varepsilon =\id_{F^{-p-1}}$ holds. In this way, one can obtain a homotopy operators $h_p$ with properties (i) and (ii) for each $p \geq 0$.

Recall that, by the Poincar\'e--Birkhoff--Witt theorem, 
$$
\{(0,1)\} \cup \{(\dsuspend (\suspend e_{i_{1}}\cdots \suspend e_{i_{l}}\otimes e_{j_{1}}\odot \cdots \odot e_{j_{k}}), 0) \mid  i_{1}\leq \cdots \leq i_{l}, \  j_{1} < \cdots < j_{k}  \}
$$
is a basis of $\big(\cU\frakg \otimes S(\frakg[1])\big)[1] \oplus \KK$. By the construction, $h((0,1))=(\dsuspend(1\otimes 1),0)$. Also, each basis vector of the form ${\bf v} =( \dsuspend(\suspend e_{i_{1}}\cdots \suspend e_{i_{l}}\otimes e_{j_{1}}\odot \cdots \odot e_{j_{k}}), 0)$ belongs to $F^{-(k+l)}$, and thus 
$$
h({\bf v}) \in F^{-(k+l)} \cap \big(\cU\frakg \otimes S^{k+1}(\frakg[1])\big)[1] \oplus \{0\} =  \big(\cU\frakg^{\leq l-1} \otimes S^{k+1}(\frakg[1])\big)[1] \oplus \{0\}.
$$
Therefore, we have 
\begin{equation}\label{eq:Ugh}
h\Big(\big(\cU\frakg^{\leq q} \otimes S(\frakg[1])\big)[1] \oplus \KK\Big) \subset  \big(\cU\frakg^{\leq q-1} \otimes S(\frakg[1])\big)[1] \oplus \{0\}.
\end{equation}

Now, we define $S(\frakg[1])\dual$-$S(\frakg[1])\dual$-subbimodules of $\Cone(\varepsilon_{\ast})$ by
\[M_{0}=\Hom_{\cU\frakg}\Big(\cU\frakg \otimes S(\frakg[1]), \big(\cU\frakg^{\leq 0} \otimes S(\frakg[1])\big)[1]\oplus \{0\} \Big)\]
 and
\[M_{q}=\Hom_{\cU\frakg}\Big(\cU\frakg \otimes S(\frakg[1]), \big(\cU \frakg^{\leq q} \otimes S(\frakg[1])\big)[1]\oplus \KK \Big)\]
for each $q>0$. Then $\{M_{q}\}_{q\geq 0}$ forms an ascending filtration of $\Cone(\varepsilon_{\ast})$ satisfying $\bigcup_{q} M_{q} = \Cone(\varepsilon_{\ast})$. Moreover, \eqref{eq:Ugh} implies that the induced contracting homotopy $\tilde{h}=h_{\ast}$ for $\Cone(\varepsilon_{\ast})$ satisfies $\tilde{h}(M_{q})\subset M_{q-1}$. Thus, by Lemma~\ref{lem:Mk}, the mapping cone $\Cone(\varepsilon_\ast)$ is pointwisely nilpotent, and thus $\rho_B$ is weakly cone-nilpotent.
\end{proof}

\subsection{Exactness of the sequences}

Now we show that the sequences \eqref{eq:KellerSeqA} and \eqref{eq:KellerSeqB} are exact for the triple $(A,X,B) = (\cU\frakg,\, \cU\frakg \otimes S(\frakg[1]), S(\frakg[1])\dual)$.

Let $V$ be a vector space concentrated at degree zero, $\{e_1, \cdots, e_d\}$ be a basis for $V[1]$, and $\{\epsilon^1, \cdots,\epsilon^d\}$ be the dual basis for $(V[1])\dual$ such that $\pair{\epsilon^j}{e_i} = \delta_i^j$. 
We denote 
$$
\omega := e_1 \odot \cdots \odot e_d \in S^d(V[1]), \quad 
\tau:= \epsilon^d \odot \cdots \odot \epsilon^1 \in S^d(V[1])\dual.
$$
To prove the exactness of \eqref{eq:KellerSeqA}, we need the following technical lemma. 

\begin{lemma}\label{lem:TopForm} 
For ${\bf x} \in S(V[1])$, $b \in S(V[1])\dual$, we have   
\begin{itemize}
\item[(i)]
$\omega \Sgvaction (\tau \Sgaction {\bf x}) = (-1)^{d- |{\bf x}|} \ {\bf x}$,
\item[(ii)]
$\tau \Sgaction ({\bf x} \Sgvaction b) = (-1)^{|b|} \  (\tau \Sgaction {\bf x}) \odot b$.
\end{itemize}
\end{lemma}
\begin{proof}
In the following computation, we use multi-index notations and Einstein summation convention. 
Let $E_i$ be the $d$-tuple with $1$ at the $i$-th component and $0$ elsewhere.  
We denote $e_I = e_{i_1} \odot \cdots \odot e_{i_k}$ and $\epsilon^J = \epsilon^{j_1} \odot \cdots \odot \epsilon^{j_l}$, where $I=E_{i_1} + \cdots + E_{i_k}$, $i_1 < \cdots < i_k$, and $J = E_{j_1}+  \cdots + E_{j_l}$, $j_1 < \cdots < j_l$. 
We say $I$ is smaller than $J$, denoted $I \leq J$, if $\{i_1, \cdots, i_k\} \subset \{j_1, \cdots, j_l\}$. If $I \leq J$, we denote $J-I$ to be the tuple associated with $\{j_1, \cdots, j_l\} \setminus \{i_1, \cdots, i_k\}$.

Let $T = E_1 + \cdots + E_d$, and $T-I = E_{\tilde\imath_1}+ \cdots + E_{\tilde\imath_{d-k}}$, $\tilde\imath_1 < \cdots < \tilde\imath_{d-k}$.  
We have
\begin{align*}
 \tau \Sgaction e_I & =  (\epsilon^d \cdots \epsilon^1) \Sgaction e_{i_1} \cdots \Sgaction e_{i_k}  = (-1)^{i_1+ \cdots + i_k -\frac{k(k+1)}{2}} \  \epsilon_{\rev}^{T - I}, 
\end{align*}
where $\epsilon_{\rev}^{T - I} = \epsilon^{\tilde\imath_{d-k}} \odot \cdots \odot \epsilon^{\tilde\imath_{1}}$. 
Thus, 
\begin{align*}
\omega \Sgvaction ( \tau \Sgaction e_I) & = (-1)^{i_1+ \cdots + i_k -\frac{k(k+1)}{2}} \   (e_1 \odot \cdots \odot e_d) \Sgvaction \epsilon^{\tilde\imath_{d-k}} \cdots \Sgvaction \epsilon^{\tilde\imath_{1}} \\
& = (-1)^{i_1+ \cdots + i_k -\frac{k(k+1)}{2}}  (-1)^{(d-\tilde\imath_1-d+k)+ \cdots + (d-\tilde\imath_{d-k}-1)} \ e_I \\
& = (-1)^{d-k} \ e_I.
\end{align*}
Since $\{e_I\}_I$ is a basis for $S( V[1])$, the first equation follows.

For the second equation, it suffices to verify it for any ${\bf x} = e_I$ and $b = \epsilon^J$. If $J \nleq I$, it is clear that 
$$
(\tau \Sgaction e_I) \odot \epsilon^J =0 = \tau \Sgaction (e_I \Sgvaction \epsilon^J).
$$
Thus, we can assume $J \leq I$. 
Let $j_p = i_{\jmath_p}$ and $I-J =  E_{i_{\tilde\jmath_1}} + \cdots + E_{i_{\tilde\jmath_{k-l}}}$. Then we have 
\begin{align*}
e_I \Sgvaction \epsilon^J & = (-1)^{l}(-1)^{(k-\jmath_1)+ \cdots + (k-\jmath_l)} e_{I-J}, \\
\tau \Sgaction (e_I \Sgvaction \epsilon^J) 
& = (-1)^{i_{\tilde\jmath_1}+ \cdots + i_{\tilde\jmath_{k-l}}} (-1)^{\jmath_1 + \cdots + \jmath_l +  \frac{k^2+k + l^2 + l}{2} } \  \epsilon_{\rev}^{T-I+J}. 
\end{align*}
Furthermore, 
\begin{align*}
\tau \Sgaction e_I & = (-1)^{(i_1 -1) + \cdots + (i_k -k)} \epsilon_{\rev}^{T-I}, \\
(\tau \Sgaction e_I) \odot \epsilon^J 
& = (-1)^{(i_1 -1) + \cdots + (i_k -k)} (-1)^{(i_{\jmath_1} -\jmath_1) + \cdots + (i_{\jmath_l} -\jmath_l +l-1)} \epsilon_{\rev}^{T-I+J} \\
& = (-1)^l \ \tau \Sgaction (e_I \Sgvaction \epsilon^J).
\end{align*}
Thus the proof is complete.
\end{proof}

\begin{lemma}
For $(A,X,B) = (\cU \frakg, \ \cU\frakg  \otimes S(\frakg[1])  , \  S( \frakg[1])\dual)$, the sequence \eqref{eq:KellerSeqA} is exact.
\end{lemma}
\begin{proof}
Define an operator
\[h_R: \Hom^r(X \otimes B^{\otimes q+1} , X) \to \Hom^r(X \otimes B^{\otimes q}, X),\]
 for each $q\geq 0$, by
$$
h_R(f)\big((u\otimes {\bf x})\,;  b_1, \cdots, b_q):= (-1)^{q+r+1}(-1)^{d-|{\bf x}|} \  f\big((u \otimes \omega)\,; (\tau \Sgaction {\bf x}), b_1, \cdots, b_q\big),
$$
where $u \in \cU \frakg$, ${\bf x} \in S(\frakg[1])$, $b_1, \cdots, b_q \in S(\frakg[1])\dual$.
By Lemma~\ref{lem:TopForm}, one can show that  
\begin{align*}
 (h_R {\hochschild^X}_R f)\big((u \otimes {\bf x})\,; b_1, \cdots, b_q \big) 
& = (-1)^{d-|{\bf x}|} f\big((u \otimes \omega)\Sgvaction (\tau \Sgaction {\bf x})\,; b_1, \cdots, b_q \big) \\
& \quad  + (-1)^{d-|{\bf x}|+1} f\big((u \otimes \omega) \,;((\tau \Sgaction {\bf x}) \odot b_1), \cdots, b_q \big) \\
& \quad + \sum_{j=1}^{q-1} (-1)^{d-|{\bf x}|+j+1} f\big((u \otimes \omega) \,; (\tau \Sgaction {\bf x}),  \cdots, b_jb_{j+1}, \cdots , b_q \big) \\
& \quad + (-1)^{d-|{\bf x}|+q+1} f\big((u \otimes \omega) \,; (\tau \Sgaction {\bf x}),  \cdots , b_{q-1} \big)\Sgvaction b_q,
\end{align*}
and 
\begin{align*}
({\hochschild^X}_R h_R f)\big((u \otimes {\bf x})\,; b_1, \cdots, b_q \big)  
& = (-1)^{d-|{\bf x}|} f\big((u \otimes \omega)\,; (\tau \Sgaction {\bf x} )\odot b_1 , \ b_2, \cdots, b_q \big) \\
& \quad + \sum_{j=1}^{q-1}(-1)^{j+d-|{\bf x}|}  f\big((u \otimes \omega)\,; (\tau \Sgaction {\bf x}), b_1, \cdots b_j b_{j+1} \cdots , b_q \big) \\
& \quad + (-1)^{q+d-|{\bf x}|} f\big((u \otimes \omega)\,; (\tau \Sgaction {\bf x}), b_1, \cdots , b_{q-1} \big) \Sgvaction b_q
\end{align*}
Thus, we have 
$$
{\hochschild^X}_R h_R +h_R {\hochschild^X}_R  = \id,
$$
and the cohomologies vanish except the zeroth cohomology. 
Since the zeroth cohomology is
$$
\ker\Big( \Hom^r(X,X) \xto{{\hochschild^X}_R} \Hom^r(X \otimes B, X) \Big) = \Hom^r_{\op B}(X,X),
$$
the proof is completed.
\end{proof}

Finally, we prove the exactness of the sequence \eqref{eq:KellerSeqB}.

\begin{lemma}
For $(A,X,B) = (\cU \frakg, \ \cU\frakg  \otimes S(\frakg[1])  , \  S(\frakg[1])\dual)$, the sequence \eqref{eq:KellerSeqB} is exact. 
\end{lemma}
\begin{proof}
In order to prove this assertion, we define an operator 
$$
h_L: \Hom^r(A^{\otimes p+1} \otimes X, X) \to \Hom^r(A^{\otimes p} \otimes X, X),
$$ 
for each $p\geq 0$, by 
$$
h_L(f)\big( a_1, \cdots, a_p;  (u \otimes {\bf x}) \big) := (-1)^{r+1} f(a_1, \cdots, a_p, u; ( 1 \otimes {\bf x}) \big),
$$
where $a_1, \cdots, a_p, u \in \cU\frakg$ and ${\bf x} \in S(\frakg[1])$. 
It is straightforward to show that  
$$
{\hochschild^X}_L h_L  + h_L {\hochschild^X}_L = \id. 
$$
Thus, 
$$
H^n\big( \Hom^r(A^{\otimes \bullet} \otimes X, X), {\hochschild^X}_L\big) = 0, 
$$
for any $n>0$ and any $r \in \ZZ$. 
Furthermore, since the zeroth cohomology is 
$$
\ker\Big( \Hom^r(X, X) \xto{{\hochschild^X}_L} \Hom^r(A \otimes X, X) \Big)  = \Hom^r_{A} (X,X),
$$
the proof is completed.
\end{proof}

Therefore, we have the following 

\begin{theorem}\label{thm:LiealgHoch}
Let $\frakg$ be a finite-dimensional Lie algebra. The triple 
$
 \big(\cU \frakg, \ \cU\frakg  \otimes S( \frakg[1]) , \  S( \frakg[1])\dual\big)
$ 
is a Keller admissible triple. Therefore, the projections  
$$
\begin{tikzcd}
 & \ar[ld,two heads,"\pi_A"'] \grHHcx^\bullet(\cU \frakg \otimes S(\frakg[1])) 
  \ar[rd,two heads,"\pi_B"] & \\
 \grHHcx^\bullet(\cU \frakg) & & \grHHcx^\bullet(S(\frakg[1])\dual)
\end{tikzcd}
$$
induce isomorphisms of Gerstenhaber algebras on cohomologies. 
\end{theorem}

It follows from Theorem~\ref{thm:LiealgHoch} that one has an isomorphism $\grHH^\bullet(\cU \frakg)\cong \grHH^\bullet(S(\frakg[1])\dual)$ of the \textit{sum} Hochschild cohomologies. The following example, however, shows that an analogous isomorphism does not exist for the \textit{product} Hochschild cohomologies, i.e. $\pdHH^\bullet(\cU \frakg)\ncong \pdHH^\bullet(S(\frakg[1])\dual)$ in general.

\begin{example} \label{ex:KellerAdmissibleSumHoch}
For the 1-dimensional Lie algebra $\frakg=\KK$, we have $\cU\frakg=S\frakg\cong\bigoplus_{n=0}^{\infty} \KK$ and $(S(\frakg[1])\dual, 0)=(\KK[x]/(x^2),0)$ as in Example~\ref{ex:SumProdHH}. Then
$$
\grHH^{0}(\cU\frakg)\cong \bigoplus_{n=0}^{\infty} \KK \cong \grHH^{0}(S(\frakg[1])\dual,0).
$$
However, the $0$-th \emph{product} Hochschild cohomology of $S(\frakg[1])\dual$ is 
$$
 \pdHH^{0}(S(\frakg[1])\dual,0) \cong \prod_{n=0}^{\infty} \KK
$$
which is \emph{not} isomorphic to $\pdHH^{0}(\cU\frakg)= \grHH^{0}(\cU\frakg)$.  
\end{example}

\section{Application to Duflo theorem}\label{sec:KD}

A dg manifold $(\cM,Q)$ is a graded manifold $\cM$ together with a homological vector field $Q$. Such a structure plays an important role in various fields of mathematics. See, for example, \cite{MR1432574,MR2275685,MR2709144,2020arXiv200601376B}. In the present paper, we consider the dg manifold $(\frakg[1], d_\frakg)$ associated with a finite-dimensional Lie algebra $\frakg$ 
whose function algebra is the Chevalley--Eilenberg dg algebra $(S(\frakg[1])\dual, d_\frakg)$.  

On the dg manifold $(\frakg[1], d_\frakg)$, one has the dg algebra $\big(\Tpoly^{\bullet}(\frakg[1]), \schouten{d_\frakg}{\argument}\big)$ of polyvector fields and the dg algebra $\big(\Dpoly^{\bullet}(\frakg[1]), \hochschild+ \gerstenhaber{d_\frakg}{\argument}\big)$ of polydifferential operators on $(\frakg[1],d_\frakg)$.  
According to \cite[Theorem~4.3]{MR3754617}, one has the Duflo--Kontsevich-type map   
\begin{equation*}\label{eq:KDiso-g[1]}
\hkr \circ \td_{\frakg[1]}^{1/2}: \big( \Tpoly^\bullet(\frakg[1]), d_T  \big) \to \big( \Dpoly^\bullet(\frakg[1]), d_D \big),
\end{equation*}
which induces a graded algebra isomorphism on their cohomologies. Here, 
\begin{gather*}
d_T := \schouten{d_\frakg}{\argument}, \\
d_D :=\hochschild + \gerstenhaber{d_\frakg}{\argument}, 
\end{gather*}
and $\td_{\frakg[1]} \in \prod_{k \geq 0}\big(\Gamma(\Lambda^k T_{\frakg[1]}\dual)\big)^k $ denotes the Todd cocycle associated with the trivial connection on $\frakg[1]$.  We refer the reader to \cite{MR3754617} for an introduction to Duflo--Kontsevich-type theorem for dg manifolds.

Let $J \in \widehat{S}(\frakg\dual)$ be the \emph{Duflo element} which is the formal power series on $\mathfrak{g}$ defined by
 $J(x)= \det\left( \frac{1-e^{-\ad_x}}{\ad_x}\right)$,  $x \in \frakg$. Its square root $J^{1/2}$ acts on $ S\frakg$ as a formal differential operator, and this map induces an operator on the Chevalley--Eilenberg cohomology $H_{\CE}^\bullet(\frakg, S\frakg )$.  
In this section, we construct isomorphisms of graded algebras
\begin{gather*}
\Phi_{T}: \hcoh^\bullet( \Tpoly(\frakg[1]), d_T ) \xto\cong H_{\CE}^\bullet(\frakg, S\frakg ), \label{eq:PhiT} \\
\Phi_{D}: \hcoh^\bullet( \Dpoly(\frakg[1]),d_D ) \xto\cong H_{\CE}^\bullet(\frakg, \cU \frakg),\label{eq:PhiD}
\end{gather*}
and prove the following:

\begin{theorem}\label{thm:KDandDuflo}
The diagram
\begin{equation}\label{diag:KDandDuflo}
\begin{tikzcd}
\hcoh^\bullet( \Tpoly(\frakg[1]), d_T ) \arrow[r, "\Todd_{\frakg[1]}^{1/2}"] \arrow[d, "\Phi_{T}"] &  \hcoh^\bullet( \Tpoly(\frakg[1]), d_T ) \arrow[r, "\hkr"] \arrow[d, "\Phi_{T}"] & \hcoh^\bullet( \Dpoly(\frakg[1]),d_D )  \arrow[d, "\Phi_{D}"]  \\
  H_{\CE}^\bullet(\frakg, S\frakg ) \arrow[r, "J^{1/2}"]  & H_{\CE}^\bullet(\frakg, S\frakg ) \arrow[r, "\pbw"]  & H_{\CE}^\bullet(\frakg, \cU \frakg) 
\end{tikzcd}
\end{equation}
commutes. 
\end{theorem}

In this way, we obtain a precise relation between the Duflo--Kontsevich-type isomorphism \cite{MR3754617} for the dg manifold $(\frakg[1],d_\frakg)$ and the Duflo--Kontsevich isomorphism \cite{MR444841,MR2062626,MR2085348} for the Lie algebra $\frakg$. 
In particular, we recover the Duflo--Kontsevich theorem:

\begin{corollary}[Kontsevich]
The map
\begin{equation*}
\pbw \circ J^{1/2}: H_{\CE}^\bullet(\frakg, S\frakg ) \to H_{\CE}^\bullet(\frakg, \cU \frakg)
\end{equation*}
is an isomorphism of graded algebras.
\end{corollary}

Note that if one replaces the \emph{direct-sum} total cohomologies by the \textit{direct-product} total cohomologies in \eqref{diag:KDandDuflo}, many issues arise. For instance, the map $\Phi_{D}$ may fail to be an isomorphism (see Example~\ref{ex:KellerAdmissibleSumHoch} and Section~\ref{sec:Dpolyg1}).

The diagram \eqref{diag:KDandDuflo} consists of two parts: left half and right half. The commutativity of the left half diagram is established in Section~\ref{sec:ToddDuflo}, and the commutativity of the right half diagram is proved in Proposition~\ref{prop:HKR=PBW}.

\subsection{Polyvector fields on $\frakg[1]$}

The dg algebra $\Tpoly^{\bullet}(\frakg[1])$ of polyvector fields on $(\frakg[1],d_{\frakg})$ consists of the graded vector space
\begin{equation*}\label{eq:Tpolyg[1]}
\Tpoly^\bullet(\frakg[1])  =\sections{\frakg[1]; S(\tangent{\frakg[1]}[-1])}\cong S(\frakg[1])\dual \otimes S\frakg
\end{equation*}
equipped with the differential $d_T = \schouten{d_\frakg}{\argument}$  and the natural multiplication 
\begin{equation*}\label{eq:Tpolyg[1]multiplication}
({\frak f}_1\otimes \tilde{\bold x}_1) \cdot ({\frak f}_2\otimes \tilde{\bold x}_2) = ({\frak f}_1\odot {\frak f}_2)\otimes (\tilde{\bold x}_1\odot \tilde{\bold x}_2),
\end{equation*}
where $\schouten{\argument}{\argument}$ is the Schouten bracket, ${\frak f}_1,{\frak f}_2 \in S(\frakg[1])\dual$ and $\tilde{\bold x}_1,\tilde{\bold x}_2 \in S\frakg$.

Let $\tilde{\Phi}_{T}:\Tpoly^{\bullet}(\frakg[1])\rightarrow \Hom(S^{\bullet}(\frakg[1]), S\frakg )$ be the map 
\[\tilde{\Phi}_{T}({\frak f} \otimes \suspend x_{1}\odot \cdots \odot \suspend x_{q}): \mathbf{y}\mapsto \pair{{\frak f}}{ \mathbf{y}}\cdot \suspend x_{1}\odot \cdots \odot \suspend x_{q} \]
where ${\frak f} \in S(\frakg[1])\dual$, $x_{i} \in \frakg[1]$, $\mathbf{y}\in S(\frakg[1])$ are homogeneous, and $\suspend:\frakg[1] \rightarrow \frakg$ is the degree-shifting map of degree $+1$. Here, the dg algebra $\Hom(S^{\bullet}(\frakg[1]), S\frakg )$ is equipped with the convolution product $\conv$ and the Chevalley--Eilenberg differential induced by the adjoint action.
See Appendix~\ref{sec:DGcoAlg} for the precise definitions.  
The following lemma follows from a direct computation.

\begin{lemma}
The map $\tilde{\Phi}_{T}:\Tpoly^{\bullet}(\frakg[1])\rightarrow \Hom(S^{\bullet}(\frakg[1]), S\frakg )$ is an isomorphism of dg algebras. In particular, the induced map 
$$
\Phi_{T}:\hcoh^{\bullet}( \Tpoly(\frakg[1]), d_T ) \rightarrow  H_{\CE}^\bullet(\frakg, S\frakg )
$$ 
is an isomorphism of graded algebra.
\end{lemma}

\subsection{Polydifferential operators on $\frakg[1]$}\label{sec:Dpolyg1}

The dg algebra $\Dpoly^{\bullet}(\frakg[1])$ of polydifferential operators on $(\frakg[1], d_\frakg)$ is the graded vector space  
\begin{equation*}\label{eq:Dpolyg[1]}
\Dpoly^\bullet(\frakg[1]) \cong \bigoplus_{p+r = \bullet} \Big(S(\frakg[1])\dual \otimes \big(S(\frakg[1])\big)^{\otimes p} \Big)^r  \cong   \bigoplus_{p+r = \bullet} \Hom^r\big((S(\frakg[1])\dual)^{\otimes p}, S(\frakg[1])\dual \big), 
\end{equation*}
equipped with the differential $d_D = \hochschild + \gerstenhaber{d_\frakg}{\argument}$ and the cup product. In other words, the dg algebra $\Dpoly^\bullet(\frakg[1])$ is isomorphic to the dg algebra $\grHHcx^\bullet(S(\frakg[1])\dual,d_{\frakg})$ of Hochschild cochains of the dg algebra $(S(\frakg[1])\dual, d_\frakg)$, and thus, 
$$
\hcoh^{\bullet}( \Dpoly(\frakg[1]), d_D ) \cong  \grHH^{\bullet}(S(\frakg[1])\dual,d_\frakg).
$$
We will omit $d_\frakg$ in the Hochschild cohomology/complex for simplicity.

By Theorem~\ref{thm:LiealgHoch}, there is an isomorphism 
$\Phi_{1}:\grHH^{\bullet}(S(\frakg[1])\dual) \to \grHH^{\bullet}(\cU \frakg)$
of graded algebras. Furthermore, it is well known that $\grHH^{\bullet}(\cU \frakg)$ is isomorphic to $H_{\CE}^{\bullet}(\frakg,\cU\frakg)$. In fact, this isomorphism is represented by the cochain map $\tilde{\Phi}_{2}:  \Hom(\cU\frakg^{\otimes \bullet}, \cU\frakg)\to \Hom(S^{\bullet}(\frakg[1]),\cU\frakg),$ 
\begin{equation}\label{eq:HH(Ug)isoHce(Ug)}
\tilde{\Phi}_{2}(f):x_{1}\odot \cdots \odot x_{p} \mapsto \sum_{\sigma\in S_p}(-1)^{\sigma } f(\suspend x_{\sigma(1)}\otimes \cdots \otimes \suspend x_{\sigma(p)}) ,
\end{equation}
for $x_1, \cdots, x_p \in \frakg[1]$ and $f \in \Hom(\cU\frakg^{\otimes p}, \cU\frakg)$. The dg algebra $\Hom(S^\bullet(\frakg[1]),\cU\frakg)$ is equipped with the convolution product $\conv$ and the Chevalley--Eilenberg differential $d_{\CE}^{\cU\frakg}$ of adjoint action. See \eqref{eq:ConvProd} and \eqref{eq:CEdiff} for the precise definitions of $\conv$ and $d_{\CE}^{\cU\frakg}$. 
The following lemma is standard \cite{MR0077480}.

\begin{lemma}[Cartan--Eilenberg]
The map $\tilde{\Phi}_{2}$ is a quasi-isomorphism of dg algebras. In particular, the induced map 
$$
\Phi_{2}: \grHH^{\bullet}(\cU\frakg) \to H_{\CE}^{\bullet}(\frakg,\cU\frakg)
$$ is an isomorphism of graded algebras.
\end{lemma}

Therefore, the map
\[\Phi_{D}:=\Phi_{2}\circ \Phi_{1}:\hcoh^{\bullet}( \Dpoly(\frakg[1]), d_D ) \cong \grHH^{\bullet}(S(\frakg[1])\dual) \rightarrow  H_{\CE}^\bullet(\frakg, \cU\frakg ),\]
is an isomorphism of graded algebras.

\subsection{Todd class and Duflo element}\label{sec:ToddDuflo}

The Todd class of a dg manifold $(\cM,Q)$ can be defined via the Atiyah class which measures the obstruction to the existence of $Q$-invariant connections. 
In the following, we present an Atiyah cocycle and a Todd cocycle of the dg manifold $(\frakg[1], d_\frakg)$, and we compare this Todd cocycle with the Duflo element of $\frakg$.  
We refer the reader to \cite{MR3319134} for a general introduction to the Atiyah and Todd classes of a dg manifold.

Let $\nabla^0$ be the trivial connection on $\frakg[1]$. The \emph{Atiyah cocycle} 
$$\at_\frakg \in S(\frakg[1])\dual \otimes (\frakg[1])\dual \otimes \End(\frakg[1]) \cong \Hom(\frakg[1]\otimes \frakg[1], \  S(\frakg[1])\dual \otimes \frakg[1])$$
associated with $\nabla^0$ is characterized by 
$$
\at_\frakg(x,y) = \suspend\inv[\suspend x, \suspend y]_\frakg
$$
for $x,y \in \frakg[1]$. 
Since $\at_\frakg(x,\argument):\frakg[1] \to \frakg[1]$ maps the odd component to the odd component, 
the \emph{Todd cocycle} $\td_{\frakg[1]} \in \prod_{k \geq 0}\big(\Gamma(\Lambda^k T_{\frakg[1]}\dual)\big)^k \cong  \widehat S \frakg\dual$ associated with $\nabla^0$ is 
$$
\td_{\frakg[1]}(x) = \Ber\left(\frac{\at_\frakg(x,\argument)}{1 - e^{-\at_\frakg(x,\argument)}} \right) = \det\left(\frac{1 - e^{-\ad_{\suspend x}}}{\ad_{\suspend x}} \right) = J(\suspend x).
$$
In other words, the Todd cocycle $\td_{\frakg[1]}$ is identified with the \emph{Duflo element} $J \in \widehat S \frakg\dual$ under the isomorphism 
\begin{equation}\label{eq:ToddisoDuflo}
\prod_{k \geq 0}\big(\Gamma(\Lambda^k T_{\frakg[1]}\dual)\big)^k \cong \prod_{k \geq 0}\big(S(\frakg[1])\dual \otimes  \Lambda^k (\frakg[1])\dual\big)^k \cong \prod_{k \geq 0}S^k\frakg\dual  =   \widehat S \frakg\dual.
\end{equation}
A Todd cocycle is closed under the coboundary operator $L_{d_\frakg}$. In our case, this means that the Todd cocycle $\td_{\frakg[1]}$ is $\frakg$-invariant. Thus, the square root $\td_{\frakg[1]}^{1/2}$ is also $\frakg$-invariant and acts on $H_{\CE}^\bullet(\frakg, S\frakg)$ by contraction.

With the isomorphism \eqref{eq:ToddisoDuflo}, one can show that the left half of the diagram \eqref{diag:KDandDuflo} commutes.

\subsection{Hochschild--Kostant--Rosenberg map and Poincar\'e--Birkhoff--Witt isomorphism} 
The last step of proving Theorem~\ref{thm:KDandDuflo} is to show that the two well-known isomorphisms --- \emph{Hochschild--Kostant--Rosenberg isomorphism} and \emph{Poincar\'e--Birkhoff--Witt isomorphism} --- are isomorphic via $\Phi_T$ and $\Phi_D$. We first recall the definitions of these two isomorphisms.

Recall that for $\mathbf{x}\in S^{k}(\frakg[1])$, the interior product 
$\iota_{\mathbf{x}}:S(\frakg[1])\dual  \rightarrow S(\frakg[1])\dual
$ 
is characterized by  
\[ \iota_{\mathbf{x}}({\frak f})= (-1)^{\degree{\mathbf{x}}\cdot \degree{{\frak f}}} \, {\frak f}(\mathbf{x}\odot \argument):S^{n-k}(\frakg[1])\rightarrow \KK, \qquad \forall\, {\frak f}\in \Hom(S^{n}(\frakg[1]),\KK) \cong S^{n}(\frakg[1])\dual , \] 
if $n\geq k$, and $\iota_{\mathbf{x}}({\frak f}) = 0$ if $n< k$.
The \emph{Hochschild--Kostant--Rosenberg map} \cite[Section~4]{MR3754617} on the graded manifold $\frakg[1]$ is the map
\begin{gather*}
\hkr: \Tpoly^{\bullet}(\frakg[1]) \to \Dpoly^{\bullet}(\frakg[1]) \cong \grHHcx^\bullet( S( \frakg[1])\dual ),  \\
\hkr\big({\frak f} \otimes (\suspend x_1 \odot \cdots \odot \suspend x_q)\big) = \frac{1}{q!} \sum_{\sigma \in S_q} {\frak f} \odot \iota_{x_{\sigma(1)}} \otimes  \cdots \otimes \iota_{x_{\sigma(q)}},
\end{gather*}
for ${\frak f}\otimes (\suspend x_{1}\odot \cdots \odot \suspend x_{q})\in S(\frakg[1])\dual \otimes S^{q}\frakg \subset \Tpoly^{\bullet}(\frakg[1])$. 
Here, the Hochschild cochain 
\[{\frak f}\odot \iota_{x_{1}}\otimes \cdots \otimes \iota_{x_{q}}: (S(\frakg[1])\dual)^{\otimes q} \to S(\frakg[1])\dual\]
 is defined by
\[{\frak f}\odot \iota_{x_{1}}\otimes \cdots \otimes \iota_{x_{q}}:(b_{1}\otimes \cdots \otimes b_{q})\mapsto (-1)^{\sum_{i=1}^{q}(q-i)\degree{b_{i}}} \, {\frak f}\odot \iota_{x_{1}}b_{1}\odot \cdots \odot \iota_{x_{q}}b_{q}\]
for $b_1, \cdots, b_q \in S(\frakg[1])\dual$. This map induces an isomorphism $\hkr: \hcoh^\bullet( \Tpoly(\frakg[1]), d_T ) \to   \grHH^\bullet(S( \frakg[1])\dual)$ of vector spaces on their cohomologies.

The \emph{Poincar\'e--Birkhoff--Witt isomorphism} is an isomorphism $\pbw:S\frakg\rightarrow \cU\frakg$, 
\begin{equation}\label{eq:pbwFormula}
\pbw( \suspend x_{1}\odot \cdots \odot \suspend x_{q})=\sum_{\sigma\in S_{n}}\frac{1}{n!}\suspend x_{\sigma(1)}\cdots \suspend x_{\sigma(q)},
\end{equation}
of $\frakg$-modules which induces an isomorphism of the Chevalley--Eilenberg complexes  
\[\pbw:\big(\Hom(S(\frakg[1]), S\frakg), d_{\CE}^{S\frakg}\, \big)  \rightarrow \big( \Hom(S(\frakg[1]), \cU\frakg), d_{\CE}^{\cU\frakg} \, \big).\]

\begin{proposition}\label{prop:HKR=PBW}
The diagram 
$$
\begin{tikzcd}
\hcoh^\bullet( \Tpoly(\frakg[1]), d_T ) \ar[d,"\Phi_{T}"'] \ar[r,"\hkr"] &  \grHH^\bullet(S( \frakg[1])\dual) \ar[d,"\Phi_D = \Phi_2 \circ \Phi_1"] \\
H_{\CE}^\bullet(\frakg, S\frakg) \ar[r,"\pbw"'] & H_{\CE}^\bullet(\frakg, \cU\frakg)
\end{tikzcd}
$$ 
commutes, where $\pbw : H_{\CE}^\bullet(\frakg, S\frakg) \to H_{\CE}^\bullet(\frakg, \cU\frakg)$ is the map induced by \eqref{eq:pbwFormula}.
\end{proposition}

The rest of the section is devoted to proving Proposition~\ref{prop:HKR=PBW} which is the last step of proving Theorem~\ref{thm:KDandDuflo}.

According to Theorem~\ref{thm:d.s.KellerThm}, the isomorphism $ \Phi_1={\pi_A}_\ast  \circ ({\pi_B}_\ast)\inv :\grHH^\bullet(S( \frakg[1])\dual) \xto\cong \grHH^\bullet(\cU \frakg)$ is induced by the surjective quasi-isomorphisms
$$
\begin{tikzcd}
\grHHcx^\bullet(\cU\frakg) & \ar[l,two heads,"\pi_A"'] \grHHcx^\bullet(\cU\frakg \otimes S( \frakg[1]))  \ar[r,two heads,"\pi_B"] & \grHHcx^\bullet(S( \frakg[1])\dual).
\end{tikzcd}
$$
Note that the inclusion $\iota_B:\grHHcx^\bullet(S( \frakg[1])\dual) \into \grHHcx^\bullet(\cU\frakg \otimes S( \frakg[1]))$ is \emph{not} a cochain map, and there is no obvious representation of $({\pi_B}_\ast)\inv$ on the cochain level. 
Thus, it is difficult to verify $\pbw \circ \Phi_T = \Phi_D \circ \hkr= \Phi_2 \circ {\pi_A}_\ast  \circ ({\pi_B}_\ast)\inv \circ \hkr$ directly.  
We solve this issue by lifting $\pi_B$ and $\hkr$ to a pullback complex.

\subsubsection*{Pullback complex}

Let $\hkr^\ast\grHHcx^\bullet(\cU\frakg \otimes S( \frakg[1]))$ be the pullback complex 
$$
\begin{tikzcd}
\hkr^\ast\grHHcx^\bullet(\cU\frakg \otimes S( \frakg[1]))  \ar[r,dashed,"\widetilde\hkr"] \ar[d,dashed,"\tilde\pi_B"'] & \grHHcx^\bullet(\cU\frakg \otimes S( \frakg[1])) \ar[d,"\pi_B"]  \\
\Tpoly^{\bullet}(\frakg[1]) \ar[r,"\hkr"']  &  \grHHcx^\bullet(S( \frakg[1])\dual).
\end{tikzcd}
$$
More precisely, 
\begin{equation*}
\hkr^\ast\grHHcx^\bullet(\cU\frakg \otimes S( \frakg[1])) = \grHHcx^\bullet(\cU\frakg) \oplus \grHHcx^\bullet(\cU\frakg,\cU\frakg \otimes S( \frakg[1]), S( \frakg[1])\dual) \oplus \Tpoly^{\bullet}(\frakg[1])
\end{equation*}
together with the differential 
$$
D = \hochschild^{A} + \hochschild^{AX} + \hochschild^X +\partial_X + (\hochschild^{XB} \circ \hkr) + \,  \schouten{d_\frakg}{\argument}.
$$

Since $\pi_B: \grHHcx^\bullet(\cU\frakg \otimes S( \frakg[1]))  \to \grHHcx^\bullet(S( \frakg[1])\dual)$ is a surjective quasi-isomorphism, the kernel $\ker(\pi_B)$ is acyclic. Also note that $\ker(\pi_{B})=\ker(\tilde\pi_{B})$ by construction. Thus, the projection
$$
\tilde\pi_B: \hkr^\ast\grHHcx^\bullet(\cU\frakg \otimes S( \frakg[1])) \to \Tpoly^{\bullet}(\frakg[1])
$$
is also a surjective quasi-isomorphism. Therefore, to prove Proposition~\ref{prop:HKR=PBW}, it suffices to show that the diagram
\begin{equation*}
\begin{tikzcd}
\hkr^\ast\grHHcx^\bullet(\cU\frakg \otimes S( \frakg[1]))  \ar[r,"\widetilde\hkr"] \ar[d,"\tilde\Phi_{T}\circ \tilde\pi_B"'] & \grHHcx^\bullet(\cU\frakg \otimes S( \frakg[1])) \ar[r,"\pi_A"] &  \grHHcx^\bullet(\cU\frakg) \ar[d,"\tilde\Phi_{2}"]     \\
\Hom(S( \frakg[1]), S\frakg) \ar[rr,"\pbw"']  &   & \Hom(S( \frakg[1]), \cU\frakg)
\end{tikzcd}
\end{equation*}
commutes up to homotopy. We prove it by constructing an explicit homotopy operator.

\subsubsection*{Homotopy operator} 
Let
\begin{align*} 
 \psi_{1} &= \tilde\Phi_{2}\circ \pi_{A}\circ \widetilde\hkr, \\
 \psi_{2} &= \pbw \circ \tilde{\Phi}_{T}\circ \tilde{\pi}_{B}.
 \end{align*}
We will construct a homotopy operator 
$h:\hkr^\ast\grHHcx^\bullet(\cU\frakg \otimes S( \frakg[1])) \rightarrow \Hom(S(\frakg[1]),\cU\frakg)$
satisfying the equation
\begin{equation}\label{eq:HptHKRPBW}
\psi_{1} - \psi_{2} = h\circ D + d_{\CE}^{\cU\frakg} \circ h.
\end{equation}

\paragraph{Notations}

In the rest of this section, we denote $A=\cU\frakg$, $B=S(\frakg[1])\dual$ and $X=\cU\frakg \otimes S(\frakg[1])$. A basis of $\frakg[1]$ is denoted by $\{e_{1},\cdots, e_{d}\}$, and its dual basis is denoted by $\{\epsilon^{1},\cdots, \epsilon^{d}\}$. The symbol $x_{i}$ is an element of $\frakg[1]$ for each $i$, $\suspend:\frakg[1]\rightarrow \frakg$ is the degree-shifting map, and $\action$ denotes the adjoint action.  
We need the technical maps 
\begin{enumerate}
\item[(i)] $\HmtpA:S^{p}(\frakg[1]) \to \cU\frakg^{\otimes p}$,
\[\HmtpA(x_{1}\odot \cdots \odot x_{p}) = \sum_{\sigma \in S_{p}} (-1)^{\sigma}\suspend x_{\sigma(1)}\otimes \cdots \otimes \suspend x_{\sigma(p)},\]
\item[(ii)] $\HmtpB{q,r}: S^{q+r}(\frakg[1])\to \cU\frakg \otimes S^{q+r}(\frakg[1]) \otimes (\frakg[1]\dual)^{\otimes q} \otimes \frakg[1]^{\otimes q}$,
\[\HmtpB{q,r}(\mathbf{x})=1\otimes \mathbf{x}\otimes \sum_{i_{1},\cdots,i_{q}}\big((\epsilon^{i_{1}}\otimes \cdots \otimes \epsilon^{i_{q}})\otimes (e_{i_{q}}\otimes \cdots \otimes e_{i_{1}})\big),\]
for $\mathbf{x}\in S^{q+r}(\frakg[1])$,
\item[(iii)] $\HmtpC: \frakg[1]^{\otimes q}\to \cU\frakg$,
\[\HmtpC(x_{1}\otimes \cdots \otimes x_{q})=\suspend x_{1} \cdots \suspend x_{q},\]
\item[(iv)]
the projection $\pr_{\cU\frakg}: \cU\frakg \otimes S(\frakg[1]) \onto \cU\frakg \otimes S^0(\frakg[1]) \cong \cU\frakg$.
\end{enumerate}

For ${\bf x}= x_1 \odot \cdots \odot x_n \in S^n(\frakg[1])$, we denote ${\bf x}^{\{i\}}:= x_1 \odot \cdots \widehat{x_i} \cdots \odot x_n$. 
Following Sweedler notation, we write the formula of the comultiplication $\Delta:S(\frakg[1])\rightarrow S(\frakg[1])\otimes S(\frakg[1])$ as 
\[\Delta(\mathbf{x})=\sum_{k}\mathbf{x}_{(1),k}\otimes \mathbf{x}_{(2),n-k}=\mathbf{x}_{(1),k}\otimes \mathbf{x}_{(2),n-k},\]
where $\mathbf{x}_{(1),k}\in S^{k}(\frakg[1])$ and $\mathbf{x}_{(2),n-k}\in S^{n-k}(\frakg[1])$.

For $f \in \Hom^{r}(A^{\otimes p}\otimes X \otimes B^{\otimes q},X)$, we define 
\begin{gather*}
h_{p,q,r}:\Hom^{r}(A^{\otimes p}\otimes X \otimes B^{\otimes q},X) \rightarrow \Hom(S^{p+q+r}(\frakg[1]),\cU\frakg)\\
h_{p,q,r}(f)=\sgn(p,q,r) \ \mu_{\cU\frakg} \circ (\pr_{\cU\frakg}\otimes \HmtpC)\circ (f\otimes 1) \circ (\HmtpA \otimes \HmtpB{q,r})\circ \Delta_{p,q+r}
\end{gather*}
where $\sgn(p,q,r)=(-1)^{qr+r+q(q+1)/2}$. Using Sweedler notation, one can write
\[h_{p,q,r} (f)(\bold x)=\sgn(p,q,r) \sum_{i_1,\cdots, i_{q}} \pr_{\cU\frakg}\circ f \Big( \HmtpA(\bold x_{(1),p})\otimes (1\otimes \bold x_{(2),q+r})\otimes (\epsilon^{i_{1}}\otimes \cdots \otimes \epsilon^{i_{q}}) \Big) \cdot (\suspend e_{i_{q}}\cdots \suspend  e_{i_{1}})\]
for $\bold x \in S^{p+q+r}(\frakg[1])$.
Now, we define the homotopy operator
\begin{equation*}
h:\hkr^\ast\grHHcx^\bullet(\cU\frakg \otimes S( \frakg[1])) \rightarrow \Hom(S(\frakg[1]), \cU\frakg)
\end{equation*}
to be the operator extending $\sum h_{p,q,r}$ by zero, i.e.\ the operator $h$ is equal to the composition
\begin{equation*}
\begin{tikzcd}
\hkr^\ast\grHHcx^\bullet(\cU\frakg \otimes S( \frakg[1])) \arrow[r, "\pr"]& 
\grHHcx(\cU\frakg, \cU\frakg \otimes S(\frakg[1]), S(\frakg[1])\dual) \arrow[r, "\sum h_{p,q,r}"] &\Hom(S(\frakg[1]), \cU\frakg).
\end{tikzcd}
\end{equation*}

\paragraph{$A$-side} 
Since $A = \cU\frakg$ is concentrated at degree zero, we have 
\[(\hochschild^{AX}f_{A})(a_{1}\otimes \cdots \otimes a_{p}\otimes x)=f_{A}(a_{1}\otimes \cdots \otimes a_{p})\cdot x\]
for  $f_{A}\in \Hom(\cU\frakg^{\otimes p},\cU\frakg) \subset\grHHcx^{\bullet}(X)$, $x \in X$ and $a_{i}\in \cU\frakg$. 
Thus, for ${\bf x} = x_1 \odot \cdots \odot x_p  \in S(\frakg[1])$,
\begin{align*}
h (\hochschild^{AX}f_{A}) (\bold x) &= \sgn(p,0,0) \ \pr_{\cU\frakg}\circ (\hochschild^{AX}f_{A})(\HmtpA(\bold x)\otimes (1\otimes 1))\\
&=  \sum_{\sigma\in S_{p}}(-1)^{\sigma} f_{A}(\suspend x_{\sigma(1)}\otimes \cdots \otimes \suspend x_{\sigma(p)}).
\end{align*} 
Furthermore, since $\psi_{2}(f_{A})=\pbw \circ \widetilde{\Phi}_{T}\circ \tilde{\pi}_{B}(f_{A})=0$ 
and 
\[\psi_{1}(f_{A})(\bold x) = (\tilde\Phi_{2}\circ \pi_{A}\circ \widetilde\hkr)(f_{A}) (\bold x) =  \sum_{\sigma \in S_p}(-1)^{\sigma} f_{A}(\suspend x_{\sigma(1)}\otimes \cdots \otimes \suspend x_{\sigma(p)}),\]
we have
\[(h\circ D + d_{\CE}^{\cU\frakg} \circ  h)(f_{A})=h (\hochschild^{AX}f_{A})=(\psi_{1} - \psi_2)(f_{A}).\]

\paragraph{$B$-side} 
For $f_{B}={\frak f} \otimes \suspend e_{j_1}\odot \cdots \odot \suspend e_{j_q} \in S^{q+r}(\frakg[1])\dual \otimes S^{q}\frakg \subset \Tpoly(\frakg[1])$ and $\bold x \in S^{q+r}(\frakg[1])$, we have 
\begin{align*}
& h( \hochschild^{XB}\circ \hkr (f_{B}))(\bold x)\\
& \quad =\sgn(0,q,r)\sum_{i_1,\cdots, i_{q}} \pr_{\cU\frakg}\circ  (\hochschild^{XB}\circ \hkr(f_{B})) \Big( (1\otimes \bold x)\otimes (\epsilon^{i_{1}}\otimes \cdots \otimes \epsilon^{i_{q}}) \Big) \cdot (\suspend e_{i_{q}} \cdots \suspend e_{i_{1}})\\
&\quad =\sgn(0,q,r) (-1)^{q+r-1+r(q+r)}\sum_{i_1,\cdots, i_{q}}\pr_{\cU\frakg}\Big((1\otimes \bold x)\Sgvaction \hkr (f_{B})(\epsilon^{i_{1}}\otimes \cdots \otimes \epsilon^{i_{q}}) \Big) \cdot (\suspend e_{i_{q}}\cdots \suspend e_{i_{1}})\\
&\quad =\sgn(0,q,r) (-1)^{q+r-1+r(q+r)+q(q+1)/2} (\bold x \Sgvaction {\frak f}) \,  \frac{1}{q!} \sum_{\sigma\in S_{q}} \suspend e_{j_{\sigma(q)}}\cdots \suspend e_{j_{\sigma(1)}}\\
&\quad =\sgn(0,q,r) (-1)^{qr+q-1+q(q+1)/2} (-1)^{q+r} {\frak f}(\bold x) \cdot \pbw(\suspend e_{j_1}\odot \cdots \odot \suspend e_{j_q}) \\
&\quad = - {\frak f}(\bold x) \cdot \pbw(\suspend e_{j_1}\odot \cdots \odot \suspend e_{j_q}).
\end{align*}
Since $\psi_{1}(f_{B}) =\tilde\Phi_{2}\circ \pi_{A}\circ \widetilde\hkr(f_{B})= 0$ and 
$$
\psi_{2}(f_{B})=\pbw\circ \tilde{\Phi}_{T}({\frak f} \otimes \suspend e_{1}\odot \cdots \odot \suspend e_{q}): \bold x \mapsto {\frak f}(\bold x) \cdot \pbw(\suspend e_{1}\odot \cdots \odot \suspend e_{q}),
$$ 
we conclude that 
\[(h\circ D+ d_{\CE}^{\cU\frakg} \circ  h )(f_{B})=h\circ \hochschild^{XB}\circ \hkr(f_{B})=(\psi_1-\psi_{2})(f_{B}).\]

\paragraph{$X$-side} 
Let $f\in \Hom^{r}(A^{\otimes p}\otimes X \otimes B^{\otimes q},X)$ and $\bold x = x_{1}\odot \cdots \odot x_{p+q+r+1} \in S^{p+q+r+1}(\frakg[1])$. To verify the homotopy equation, we need to compute the following 4 terms.

{\bf Term I:} 
\begin{align*}
d_{\CE}^{\cU\frakg}(hf)(\bold x)&=\left( \sum_{i=1}^{p+q+r+1} (-1)^{i+p+q+r}(\suspend x_{i})\action(hf)(\bold x^{\{i\}})\right) -(-1)^{p+q+r}(hf)(\partial_\frakg \bold x)\\
&=\left( \sum_{i=1}^{p+q+r+1}(-1)^{i+p+q+r}\sgn(p,q,r) (\mathfrak{D}^{i}_{1}-\mathfrak{D}^{i}_{2})\right)  -(-1)^{p+q+r}\sgn(p,q,r)(\mathfrak{F}_{1}+\mathfrak{F}_{2}),
\end{align*}
where
\begin{align*}
\mathfrak{D}^{i}_{1} &=\sum_{i_1,\cdots, i_{q}} \suspend x_{i}\cdot \left(\pr_{\cU\frakg}\circ f \Big( \HmtpA(\bold x_{(1),p}^{\{i\}})\otimes (1\otimes \bold x_{(2),q+r}^{\{i\}})\otimes (\epsilon^{i_{1}}\otimes \cdots \otimes \epsilon^{i_{q}}) \Big)\right) \cdot (\suspend  e_{i_{q}}\cdots \suspend  e_{i_{1}})\\
 \mathfrak{D}^{i}_{2} &=\sum_{i_1,\cdots, i_{q}} \left(\pr_{\cU\frakg}\circ f \Big( \HmtpA(\bold x_{(1),p}^{\{i\}})\otimes (1\otimes \bold x_{(2),q+r}^{\{i\}})\otimes (\epsilon^{i_{1}}\otimes \cdots \otimes \epsilon^{i_{q}}) \Big)\right) \cdot (\suspend  e_{i_{q}}\cdots \suspend  e_{i_{1}})\cdot \suspend x_{i} \\
 \mathfrak{F}_{1} &=\sum_{i_1,\cdots, i_{q}}\pr_{\cU\frakg}\circ f \Big(\HmtpA(\partial_\frakg\bold x_{(1),p+1})\otimes (1\otimes \bold x_{(2),q+r})\otimes (\epsilon^{i_{1}}\otimes \cdots \otimes \epsilon^{i_{q}}) \Big)\cdot (\suspend  e_{i_{q}}\cdots \suspend  e_{i_{1}}) \\
\mathfrak{F}_{2}& =(-1)^{p}\sum_{i_1,\cdots, i_{q}}\pr_{\cU\frakg}\circ f\Big(\HmtpA(\bold x_{(1),p})\otimes (1\otimes \partial_\frakg\bold x_{(2),q+r+1})\otimes (\epsilon^{i_{1}}\otimes \cdots \otimes \epsilon^{i_{q}}) \Big)\cdot (\suspend  e_{i_{q}}\cdots \suspend  e_{i_{1}}).
\end{align*}
{\bf Term II:}
\begin{multline*}
h({\hochschild^{X}}_{L}f)(\bold x)
=\sgn(p+1,q,r) (-1)^{p+q+r} \cdot \\
 \left \{ \left( \sum_{i=1}^{p+q+r+1}(-1)^{i+1} \mathfrak{A}^{i}_{1} \right) + \mathfrak{A_{\text{mid}}}+\left( (-1)^{p+1}\sum_{i=1}^{p+q+r+1}(-1)^{i+1+p} \mathfrak{A}^{i}_{2}\right)\right\},
\end{multline*}
where 
\begin{align*}
\mathfrak{A}^{i}_{1} &=\sum_{i_1,\cdots, i_{q}} \suspend x_{i}\cdot \left(\pr_{\cU\frakg}\circ f \Big( \HmtpA(\bold x_{(1),p}^{\{i\}})\otimes (1\otimes \bold x_{(2),q+r}^{\{i\}})\otimes (\epsilon^{i_{1}}\otimes \cdots \otimes \epsilon^{i_{q}}) \Big)\right) \cdot (\suspend  e_{i_{q}}\cdots \suspend  e_{i_{1}})\\
\mathfrak{A}^{i}_{2}&=\sum_{i_1,\cdots, i_{q}}\pr_{\cU\frakg}\circ f \Big( \HmtpA(\bold x_{(1),p}^{\{i\}})\otimes ( \suspend x_{i}\otimes \bold x_{(2),q+r}^{\{i\}})\otimes (\epsilon^{i_{1}}\otimes \cdots \otimes \epsilon^{i_{q}}) \Big) \cdot (\suspend  e_{i_{q}}\cdots \suspend  e_{i_{1}})\\
\mathfrak{A}_{\text{mid}} &= \sum_{i_1,\cdots, i_{q}}\sum_{j=1}^{p}(-1)^{j} \cdot \\
& \qquad \quad \pr_{\cU\frakg}\circ f \Big( \left(  \mu _{j} \circ \HmtpA(\bold x_{(1),p+1}) \right) \otimes (1\otimes \bold x_{(2),q+r})\otimes (\epsilon^{i_{1}}\otimes \cdots \otimes \epsilon^{i_{q}}) \Big) \cdot (\suspend  e_{i_{q}}\cdots \suspend  e_{i_{1}}),
\end{align*}
and the map $\mu _{j}:\cU\frakg^{\otimes p+1} \rightarrow \cU\frakg^{\otimes p}$, $j=1,\cdots, p$, is defined by
\[\mu _{j}(u_{1}\otimes \cdots \otimes u_{p+1})=u_{1}\otimes \cdots \otimes u_{j}u_{j+1}\otimes \cdots \otimes u_{p+1}.\]
{\bf Term III:} 
\begin{align*}
&h({\hochschild^{X}}_{R}f)(\bold x) =\sgn(p,q+1,r)(-1)^{p+q+r} (\mathfrak{B}_{1}+\mathfrak{B}_{2}+\mathfrak{B}_{3}),
\end{align*}
where
\begin{align*}
\mathfrak{B}_{1} 
&=-\sum_{ i_1,\cdots, i_{q}}\sum_{i=1}^{p+q+r+1}(-1)^{q+r-i} \cdot \\
& \quad \qquad  \pr_{\cU\frakg}\circ f \Big( \HmtpA(\bold x^{\{i\}}_{(1),p})\otimes (1\otimes \bold x^{\{i\}}_{(2),q+r})\otimes (\epsilon^{i_{1}}\otimes \cdots \otimes \epsilon^{i_{q}}) \Big) \cdot (\suspend  e_{i_{q}}\cdots \suspend  e_{i_{1}}\suspend x_{i})
\\
\mathfrak{B}_{2}& =
\sum_{ i_1,\cdots, i_{q+1}}\sum_{k=1}^{q}(-1)^{p+1+k} \cdot \\
& \quad \qquad   \pr_{\cU\frakg}\circ f \Big( \HmtpA(\bold x_{(1),p})\otimes (1\otimes \bold x_{(2),q+r+1})\otimes (\epsilon^{i_{1}}\otimes \cdots \epsilon^{i_{k}}\odot \epsilon^{i_{k+1}} \cdots \otimes \epsilon^{i_{q+1}}) \Big) \cdot (\suspend  e_{i_{q+1}}\cdots \suspend  e_{i_{1}})
 \\
\mathfrak{B}_{3} & =\sum_{ i_1,\cdots, i_{q},j}(-1)^{p+q}  \cdot \\
& \quad \qquad   \pr_{\cU\frakg} \left \{ f \Big( \HmtpA(\bold x_{(1),p})\otimes (1\otimes \bold x_{(2),q+r+1})\otimes (\epsilon^{i_{1}}\otimes \cdots \otimes \epsilon^{i_{q}}) \Big)\Sgvaction \epsilon^{j}\right \} \cdot (\suspend  e_{j}\suspend  e_{i_{q}}\cdots \suspend  e_{i_{1}})
\end{align*}
{\bf Term IV:} 
\begin{align*}
h(\partial_{X}f)(\bold x)
&=\sgn(p,q,r+1) (\mathfrak{I}_{1}-(-1)^{r}(\mathfrak{I}_{2}+\mathfrak{I}_{3})),
\end{align*}
where
\begin{align*}
\mathfrak{I}_{1} & =\sum_{i_1,\cdots, i_{q}} \pr_{\cU\frakg}\circ d_{X}\circ f \Big( \HmtpA(\bold x_{(1),p})\otimes (1\otimes \bold x_{(2),q+r+1})\otimes (\epsilon^{i_{1}}\otimes \cdots \otimes \epsilon^{i_{q}}) \Big) \cdot (\suspend  e_{i_{q}}\cdots \suspend  e_{i_{1}})\\
\mathfrak{I}_{2}& =\sum_{i_1,\cdots, i_{q}} \pr_{\cU\frakg}\circ f \Big( \HmtpA(\bold x_{(1),p})\otimes d_{X}(1\otimes \bold x_{(2),q+r+1})\otimes (\epsilon^{i_{1}}\otimes \cdots \otimes \epsilon^{i_{q}}) \Big) \cdot (\suspend  e_{i_{q}}\cdots \suspend  e_{i_{1}})\\
\mathfrak{I}_{3}& =\sum_{i_1,\cdots, i_{q}}\sum_{k=1}^{q}(-1)^{q+r+k} \cdot \\
& \qquad \quad   \pr_{\cU\frakg}\circ  f \Big( \HmtpA(\bold x_{(1),p})\otimes (1\otimes \bold x_{(2),q+r+1})\otimes (\epsilon^{i_{1}}\otimes \cdots \otimes d_{\frakg}\epsilon^{i_{k}}\otimes \cdots \otimes \epsilon^{i_{q}}) \Big) \cdot (\suspend  e_{i_{q}}\cdots \suspend  e_{i_{1}}).
\end{align*}

Note that we have the following equations 
\begin{gather}
\mathfrak{D}^{i}_{1}=\mathfrak{A}^{i}_{1} \nonumber \\
\sum_{i}(-1)^{q+r+i+1}\mathfrak{D}^{i}_{2}=\mathfrak{B}_{1} \nonumber  \\
\mathfrak{F}_{1}=\mathfrak{A}_{\text{mid}} \nonumber \\
 \mathfrak{I}_{2}=(-1)^{p}\mathfrak{F}_{2}+\sum_{i=1}^{p+q+r+1}(-1)^{i+p+1}\mathfrak{A}^{i}_{2} \nonumber \\
\mathfrak{B}_{2}=(-1)^{p+q+r}\mathfrak{I}_{3} \nonumber\label{eq:12}\\
\mathfrak{B}_{3}=(-1)^{p+q+1}\mathfrak{I}_{1} \nonumber\label{eq:13}
\end{gather}
for any $\bold x \in S^{p+q+r+1}(\frakg[1])$.  
Here, the last second equation is obtained by 
\[(d_{\frakg}\epsilon^{i})\otimes \suspend  e_{i} = \sum_{a<b}\epsilon^{a}\odot \epsilon^{b}\otimes [\suspend  e_{a},\suspend  e_{b}]=\sum_{a<b}\left(\epsilon^{a}\odot \epsilon^{b}\otimes \suspend  e_{a}\suspend  e_{b}+\epsilon^{b}\odot \epsilon^{a}\otimes \suspend  e_{b}\suspend  e_{a}\right) \]
which is induced by 
\[
[\suspend  e_{a},\suspend  e_{b}]=c_{ab}^{i}\suspend  e_{i},\quad 
 d_{\frakg}\epsilon^{i}=c_{ab}^{i}\epsilon^{a}\odot \epsilon^{b},
\]
and the last equation is obtained by 
\[d_{X}(u\otimes  e_{i})=u(\suspend  e_{i})\otimes 1 = -(u(\suspend  e_{i})\otimes  e_{i})\Sgvaction \epsilon^{i}.\]
Since $(\psi_1 - \psi_2)(f)=0$, the homotopy equation \eqref{eq:HptHKRPBW} is equivalent to the system
\begin{gather*}
\sgn(p+1,q,r)=\sgn(p,q,r),\\
\sgn(p,q+1,r)=(-1)^{q+r+1}\sgn(p,q,r),\\
\sgn(p,q,r+1)=(-1)^{q+1}\sgn(p,q,r),
\end{gather*} 
and the sign function 
\[\sgn(p,q,r)=(-1)^{qr+r+q(q+1)/2}\]
is a solution of this system. Therefore, we conclude that 
$$
(h\circ D + d_{\CE}^{\cU\frakg} \circ  h)(f) = 0 = (\psi_1 - \psi_2)(f).
$$
This completes the proof of Proposition~\ref{prop:HKR=PBW}, and thus the proof of Theorem~\ref{thm:KDandDuflo} is also complete.

\appendix

\section{Differential graded coalgebras and  comodules}\label{sec:DGcoAlg}

In this appendix, we summarize the necessary facts about dg coalgebras and dg comodules.  We refer the reader to \cite{MR1786197, MR2954392} for a general introduction to coalgebras.
 
A \textbf{graded coalgebra} $(C,\Delta,\epsilon)$ is a graded vector space $C$, equipped with degree-preserving linear maps $\Delta: C\rightarrow C\otimes C$, called \textbf{coproduct}, and $\epsilon:C\rightarrow \KK$, called \textbf{counit}, such that
\begin{enumerate}
\item $(\Delta\otimes \id_{C})\circ \Delta = (\id_{C}\otimes \Delta)\circ \Delta$,
\item $\mu_{\KK,C} \circ (\epsilon\otimes \id_{C})\circ \Delta = \id_{C} = \mu_{C,\KK} \circ (\id_{C}\otimes \epsilon)\circ \Delta,$
\end{enumerate}
where the ground field $\KK$ is considered as a graded vector space concentrated at degree zero, and $\mu_{\KK,C}, \mu_{C,\KK}$ are the scalar multiplications. 
A graded coalgebra $(C,\Delta,\epsilon)$ is said to be \textbf{cocommutative} if $\tw \circ \Delta = \Delta$, where $\tw: C \otimes C \to C\otimes C$ is defined by $\tw(x \otimes y) = (-1)^{|x||y|} y \otimes x$.
A \textbf{coderivation} of degree $i$ on a graded coalgebra $(C,\Delta, \epsilon)$  is a linear map $Q: C \to C$ of degree $i$ such that 
$$
\Delta\circ Q = (\id_{C}\otimes \, Q+ Q\otimes \id_{C})\circ \Delta \, ,
$$
the space of coderivations of degree $i$ on $C$ is denoted by $\coDer^i(C)$, and $\coDer(C):= \bigoplus_i \coDer^i(C)$. 

A \textbf{dg coalgebra} $(C,\Delta,\epsilon, d_{C})$ is a graded coalgebra $(C,\Delta, \epsilon)$ together with a coderivation $d_{C}:C\rightarrow C$ of degree +1 such that $d_C \circ d_C =0$.

\subsection{Convolution algebras}\label{sec:ConvAlg}

Let $(A, d_A)$ be a dg algebra with unit $1_A$, and $(C,  d_C)$ be a dg coalgebra with counit $\epsilon$. The \textbf{convolution product} $\conv: \Hom(C,A) \times \Hom(C,A) \to \Hom(C,A)$ is the multiplication on the space $\Hom(C,A)$ of $\KK$-linear maps defined by
\begin{equation}\label{eq:ConvProd}
f\conv g := \mu\circ (f\otimes g)\circ \Delta,
\end{equation}
where $f,g \in \Hom(C,A)$, $\mu$ is the multiplication on $A$, and $\Delta$ is the comultiplication on $C$. One can check that  $\big(\Hom(C,A),\conv \, \big)$ is a graded algebra with the unit $1_{\Hom(C,A)} \in \Hom(C,A)$, where 
$$
1_{\Hom(C,A)}(x):= \epsilon(x) \cdot 1_A.
$$ 
Furthermore, one can show that the linear map $d_{\Hom(C,A)}:\Hom^{\bullet}(C,A)\rightarrow \Hom^{\bullet +1}(C,A)$, 
\[d_{\Hom(C,A)}(f):=d_{A}\circ f -(-1)^{\degree{f}}f\circ d_{C},\]
is a derivation of degree one satisfying $(d_{\Hom(C,A)})^2 =0$, and thus the triple $\big(\Hom(C,A), \conv, d_{\Hom(C,A)} \big)$ is a dg algebra which is referred as the \textbf{convolution dg algebra}.

\begin{definition}
Let $V$ be a graded vector space. The \textbf{symmetric coalgebra} over $V$ is the graded vector space $S V := \bigoplus_{n =0}^\infty V^{\odot n}$ whose  counit is the projection $\epsilon_S : S V \onto S^0 V \cong \KK$, and 
whose coproduct 
$\Delta_S$ is determined by 
\begin{gather*}
\Delta_S(1) = 1 \otimes 1, \qquad \Delta_S(v) = 1 \otimes v + v \otimes 1,  \\
\Delta_S({\bf x} \odot {\bf y}) = \Delta_S({\bf x}) \odot \Delta_S({\bf y}) ,   
\end{gather*}
for any $v \in V$, ${\bf x},{\bf y} \in S V$.
\end{definition}

More precisely, the coproduct $\Delta_S$ can be computed by the formula
\begin{equation*}
\Delta_S(v_1 \odot \cdots \odot v_n) = \sum_{i=0}^n \sum_{\sigma \in S_{i,n-i}} \varepsilon \cdot (v_{\sigma(1)} \odot \cdots \odot v_{\sigma(i)}) \otimes (v_{\sigma(i+1)} \odot \cdots \odot v_{\sigma(n)}),
\end{equation*}
where $\varepsilon = \pm 1$ is determined by the Koszul sign convention. 
Also note that the coalgebra $(SV,\Delta_S,\epsilon_S)$ is cocommutative.

\begin{example}
Let $\frakg$ be a Lie algebra. We have the graded coalgebra $S(\frakg[1])$. 
The Lie bracket $[\argument, \argument]_\frakg $ induces a (graded) symmetric operation 
$$
\suspend\inv \circ [\argument, \argument]_\frakg \circ (\suspend \otimes \suspend) : S^2 (\frakg[1]) \to \frakg[1], \quad \, x \odot y \mapsto - \suspend\inv [\suspend x, \suspend y]_\frakg
$$ 
of degree one, where $\suspend: \frakg[1] \to \frakg$ is the degree-shifting map. 
This operation induces a degree-one coderivation $\partial_{\frakg}:S(\frakg[1]) \to S(\frakg[1])$ defined by
\begin{equation*}
\partial_{\frakg}(x_1\odot \cdots\odot x_n)=  \sum_{i<j} (-1)^{i+j} \, \suspend\inv[ \suspend x_i, \suspend x_j]_\frakg\odot x_1\odot \cdots \widehat{x_i}\cdots \widehat{x_j}\cdots \odot x_n \, ,
\end{equation*}
for $x_1, \cdots, x_n \in \frakg[1]$, and one obtain the dg coalgebra $(S(\frakg[1]) , \partial_{\frakg} )$.

Let $A=\KK$ be the trivial dg algebra with zero differential. The associated convolution dg algebra $\Hom(S(\frakg[1]) ,\KK ) \cong   S(\frakg[1])\dual$ is equipped with the differential 
\begin{equation}\label{eq:LieAlgDiff}
d_\frakg (f)=-(-1)^{\degree{f}}f\circ  \partial_\frakg \, .
\end{equation}   
It is straightforward to show that the convolution product $\conv$ coincides with the canonical multiplication on $S(\frakg[1])\dual$ in the sense that 
$$
(f\conv g)(\bold x) = \pair{f \odot g}{\bold x}
$$
for $f,g \in \Hom(S(\frakg[1]) ,\KK ) \cong  S(\frakg[1])\dual$ and $\bold x \in S(\frakg[1])$. 
\end{example}

\subsection{Twisting cochains}

Let $(C,d_C)$ be a dg coalgebra, and $(A,d_A)$ be a dg algebra. An element $\tau\in \Hom^{1}(C,A)$ of degree $+1$ is called a \textbf{twisting cochain} if it satisfies the Maurer--Cartan equation 
\[d_{\Hom(C,A)}(\tau)+\tau \star \tau =0\]
in the convolution algebra $\big(\Hom(C,A), \conv, d_{\Hom(C,A)} \big)$. 
The \textbf{twisted tensor product} $A\otimes_{\tau}C$ is a dg vector space whose underlying space is the tensor product $A\otimes C$, and the differential $d_{\tau}$ is defined by 
\[d_{\tau}=d_{A}\otimes \id_{C} + \id_{A}\otimes d_{C} - (\mu \otimes \id_{C})(\id_{A}\otimes \tau \otimes \id_{C})(\id_{A}\otimes \Delta).\]
Note that the twisted tensor product $A\otimes_{\tau}C$ is a left dg $(A,d_{A})$-module. 
Let $(M,d_{M})$ be a left dg $(A,d_{A})$-module. The space $\Hom_{A}(A\otimes_{\tau}C, M)$ is equipped with a canonical differential $d_\tau$ induced by the dg structures. This dg vector space $(\Hom_{A}(A\otimes_{\tau}C, M), d_\tau)$ will be denoted by $\Hom^{\tau}(C,M)$ and will be referred as the \textbf{twisted Hom space}. See, for example, \cite{MR3338683,MR2954392, lefevrehasegawa:tel-00007761}.

Let $\frakg$ be a Lie algebra. We have the dg coalgebra $(S(\frakg[1]), \partial_{\frakg})$ and the dg algebra $(\cU\frakg,0)$. One can show that the map $\tau:S(\frakg[1]) \to \cU\frakg$ defined by the composition 
\begin{equation*}
\begin{tikzcd}
S(\frakg[1])\arrow[r, "\pr_{\frakg[1]}"'] \arrow[rrr,bend left=15,"\tau"]& 
\frakg[1] \arrow[r, "-\suspend"'] &\frakg \arrow[r, hookrightarrow] & \cU\frakg
\end{tikzcd}
\end{equation*}
is a twisting cochain. Here, $\suspend:\frakg[1]\to \frakg$ is the degree shifting map, and $\frakg\into \cU\frakg$ is the natural embedding. The twisted tensor product $\cU\frakg \otimes_{\tau} S(\frakg[1])$ coincides with the Chevalley--Eilenberg chain complex $(\cU\frakg\otimes S(\frakg[1]), d_X)$ described in Section~\ref{sec:KellerTrip-LieAlg}.

Let $(B,d_{B})$ be a dg algebra. We say $(B,d_{B})$ is a \textbf{dg $\frakg$-algebra} if it is endowed with an infinitesimal action of $\frakg$ (i.e.\  a Lie algebra morphism $\frakg\to \Der^{0}(B)$) such that
\begin{equation*}
d_B(\suspend x \cdot b) = \suspend x \cdot d_B(b),
\end{equation*}
for $\suspend x \in \frakg$, $b \in B$. 
Since a dg $\frakg$-algebra $(B,d_B)$ is also a dg $\cU\frakg$-module, the graded vector space $\Hom(S(\frakg[1]),B) \cong \Hom_{\cU\frakg}(\cU\frakg\otimes S(\frakg[1]),B)$ is endowed with the differential $d_\tau$ and the convolution product $\conv$. 

\begin{proposition}\label{prop:TwistedConv}
Let $(B,d_{B})$ be a dg $\frakg$-algebra. The graded vector space $\Hom(S(\frakg[1]),B)$, equipped with the differential $d_\tau$ and the convolution product $\conv$, is a dg algebra.
\end{proposition}
\begin{proof}
Since $\big(\Hom(S(\frakg[1]),B),\conv,d_{\Hom(S(\frakg[1]),B)}\big)$ is a dg algebra, it suffices to show the compatibility of the convolution product $\conv$ and $d_{\tau}-d_{\Hom(S(\frakg[1]),B)}$. This is a consequence of the assumption $\frakg$ acts on $B$ by derivations.
\end{proof}

In the setting of Proposition~\ref{prop:TwistedConv}, the differential $d^{\tau}$ is also written as $d_{\CE}^{B}$, called the {\bf Chevalley--Eilenberg differential}. More explicitly, 
\begin{equation}\label{eq:CEdiff}
\begin{split}
d_{\CE}^{B}(f)(x_1\odot \cdots\odot x_n)&=
  \sum_{i}(-1)^{i+\degree{f}} \, \suspend x_i \cdot   f(x_{1}\odot \cdots \widehat{x_i} \cdots \odot x_{n})\\
&\quad +(d_B\circ f)(x_1 \odot \cdots\odot x_n) \\
&\quad \quad  - (-1)^{\degree{f}}  \sum_{i<j}(-1)^{i+j} f(\suspend\inv [\suspend x_i, \suspend x_j]_\frakg \odot x_1 \odot \cdots \widehat{x_i}\cdots \widehat{x_j}\cdots \odot x_n)   ,
\end{split}
\end{equation}
for $f \in \Hom(S(\frakg[1]), B)$ and $x_1, \cdots, x_n \in \frakg[1]$.

\begin{remark}
Let $(B,d_B) = (B,0)$ be a dg $\frakg$-algebra with the zero differential. 
There is an alternative formulation of the Chevalley--Eilenberg differential  
$\tilde d_{\CE}:\Hom(\Lambda^{\bullet}\frakg, B) \to \Hom(\Lambda^{\bullet+1}\frakg, B)$ given by
\begin{equation*}
\begin{split}
\tilde d_{\CE}(\tilde f)( \suspend x_1 \wedge \cdots \wedge \suspend x_{n}) & := \sum_{i=1}^n (-1)^{i+1} \suspend x_i \cdot \tilde f(\suspend x_1 \wedge \cdots \widehat{\suspend x_i} \cdots \wedge \suspend x_n) \\
& \quad +  \sum_{i<j} (-1)^{i+j} \tilde f([\suspend x_i,\suspend x_j]_\frakg \wedge \suspend x_1 \wedge \cdots \widehat{\suspend x_i} \cdots \widehat{\suspend x_j} \cdots \wedge \suspend x_n )\, .
\end{split}
\end{equation*} 
The space $\Hom(\Lambda^{\bullet}\frakg, B)$ is equipped the product $\tilde\conv:\Hom(\Lambda^{\bullet}\frakg, B) \times \Hom(\Lambda^{\bullet}\frakg, B) \to \Hom(\Lambda^{\bullet}\frakg, B)$, 
$$
(\tilde f \ \tilde\conv \ \tilde g)(\suspend x_1\wedge \cdots \wedge \suspend x_{n+m}):= \sum_{\sigma \in S_{n,m}} (-1)^\sigma \ \tilde f(\suspend x_{\sigma(1)} \wedge \cdots \wedge \suspend x_{\sigma(n)}) \tilde g(\suspend x_{\sigma(n+1)} \wedge \cdots \wedge \suspend x_{\sigma(n+m)}) \, .
$$
The triple $\big(\Hom(\Lambda^{\bullet}\frakg, B), \ \tilde\conv \ , \tilde d_{\CE}\big)$ is a dg algebra.

Let $\eta: \Hom(\Lambda^{\bullet}\frakg, B) \to \Hom(S^\bullet(\frakg[1]), B)$ be the map 
$$
\eta(\tilde f)(x_1\odot \cdots \odot x_n) := (-1)^{1+ \cdots +n}\, \tilde f(\suspend x_1 \wedge \cdots \wedge \suspend x_n)
$$
for $\tilde f \in \Hom(\Lambda^{n}\frakg, B)$, $x_1, \cdots, x_n \in \frakg[1]$. 
One can check that  
$$
\eta : \big(\Hom(\Lambda^{\bullet}\frakg, B), \ \tilde\conv \ , \tilde d_{\CE}\big) \to \big(\Hom^{\tau}(S^{\bullet}(\frakg[1]), B), \conv  \big)
$$
is an isomorphism of dg algebras. 
\end{remark}

\begin{remark}
Since the reduced Chevalley--Eilenberg chain complex $\cU\frakg \otimes_{\tau} S(\frakg[1])/\KK$ is acyclic, by a theorem \cite[Theorem~2.3.1]{MR2954392} of twisting cochains, one has a quasi-isomorphism from the cobar complex of $S(\frakg[1])$ to $\cU \frakg$. This quasi-isomorphism induces a map 
$\grHHcx^\bullet (S(\frakg[1])\dual,d_\frakg) \to \Hom^\bullet (S(\frakg[1]), \cU\frakg)$ which coincides with the map in the proof of \cite[Theorem~4.10]{MR2816610}. See also \cite{MR4237031}.
\end{remark}

\subsection{Differential graded comodules and cogenerators}\label{sec:coGencoMod} 

Let $(C, \Delta, \epsilon)$ be a graded coalgebra. A \textbf{(right) graded comodule} $(M,\phi_{M})$ over $C$ is a graded vector space $M$, equipped with a linear map $\phi_{M}:M\rightarrow M\otimes C$ of degree zero such that
\begin{enumerate}
\item[(i)] $(\phi_{M}\otimes \id_{C})\circ \phi_{M}=(\id_{M}\otimes \Delta)\circ \phi_{M}$;
\item[(ii)] $\mu_{M,\KK} \circ (\id_{M}\otimes \, \epsilon)\circ \phi_{M}=\id_{M},$
\end{enumerate}
where $\mu_{M,\KK}: M \otimes \KK \to M$ is the scalar multiplication. 
Let $(M,\phi_M)$ and $(N,\phi_N)$ be comodules over $C$. A \textbf{morphism of comodules} is a linear map $\Psi \in \Hom_{\KK}^\bullet(M, N)$  such that 
\begin{equation}\label{eq:coModMorAxiom}
(\Psi\otimes \id_{C})\circ \phi_{M} = \phi_{N}\circ \Psi.
\end{equation}
We denote by $\coHom_C(M,N)$ the space of morphisms of  comodules from $(M,\phi_M)$ to $(N,\phi_N)$.

A \textbf{(right) dg comodule} $(M,\phi_{M}, d_{M})$ over a dg coalgebra $(C,\Delta,\epsilon, d_{C})$ is a (right) graded comodule over $C$, together with a linear map $d_{M}:M\rightarrow M$ of degree +1 such that $d_{M} \circ d_{M} =0$ and
\begin{equation*}\label{eq:coModDer}
\phi_{M}\circ d_{M}=(d_{M}\otimes \id_{C} + \id_{M}\otimes \, d_{C})\circ \phi_{M}.
\end{equation*}

\begin{example}\label{ex:TwistedTensorComodule}
Let $(C,\Delta, d_C)$ be a dg coalgebra and $(A,\mu, d_{A})$ be a dg algebra. If $\tau:C\to A$ is a twisting cochain, the twisted tensor product $A\otimes_{\tau}C$ equipped with  $\id_A \otimes \Delta:A \otimes C \to A \otimes C \otimes C$ is a right dg comodule over $C$. In particular, the Chevalley--Eilenberg chain complex $(\cU\frakg\otimes S(\frakg[1]),d_X)$ is a right dg $S(\frakg[1])$-comodule.
\end{example}

\begin{lemma}
Let $(M, \phi_M, d_M)$ and $(N, \phi_N, d_N)$ be dg comodules.  
The space $\coHom_C(M,N)$ is a dg vector subspace of $\Hom_\KK(M,N)$ whose differential $\partial$ is defined by 
$$
\partial(\Psi):= d_N \circ \Psi - (-1)^{|\Psi|} \Psi \circ d_M.
$$
\end{lemma}

Let $(C,\Delta,\epsilon, d_C)$ be a dg coalgebra, and let $(A,\mu,1_A , d_A) = \big(\Hom(C,\KK), \conv, 1_{\Hom(C,\KK)}, d_{\Hom(C,\KK)} \big)$ be the convolution dg algebra.
Let $(M,\phi_{M}, d_M)$ be a right dg comodule over $C$. We define the action map $\rho_{M}:A\otimes M \rightarrow M$ by
\begin{equation*}
\rho_{M}(f\otimes m):= \mu_{M,\KK} \circ (\id_{M}\otimes f)\circ \phi_{M}(m),
\end{equation*}
where $m\in M$, $f \in A = \Hom(C,\KK)$.

\begin{proposition}\label{prop:coModv.s.Mod}
Let $(M, \phi_M)$ be a right dg comodule over a dg coalgebra $C$. 
The triple $(M,\rho_M, d_M)$ is a left dg module over the convolution dg algebra $A = \Hom(C, \KK)$. 

Furthermore, let $(N,\phi_N,d_N)$ be another right dg comodule over $C$, and $(N,\rho_N,d_N)$ be the associated dg module over $A$. A linear map $\Psi:M \to N$ is a morphism of right comodules over $C$ if and only if $\Psi$ is a morphism of left modules over $A$, i.e.\ $\coHom_C(M,N) = \Hom_A(M,N)$.  
\end{proposition}
\begin{proof}
The first part of the proposition can be shown by a direct computation. See \cite[Proposition~2.2.1]{MR1786197}. We only prove the second part here. 
Let $(N,\phi_N,d_N)$ be another right dg comodule over $C$, and $\Psi:M \to N$ be a $\KK$-linear map. If $\Psi$ is a comodule morphism over $C$, then 
\begin{align*}
\Psi \circ \rho_M(f \otimes m) & = \Psi \circ\mu_{M,\KK} \circ (\id_M \otimes f) \circ \phi_M(m) \\
& = (-1)^{|f||\Psi|} \mu_{N,\KK} \circ (\id_N \otimes f) \circ (\Psi \otimes \id_C) \circ \phi_M(m) \\
& = (-1)^{|f||\Psi|} \mu_{N,\KK} \circ (\id_N \otimes f)  \circ \phi_N \circ \Psi(m) \\
& = (-1)^{|f||\Psi|} \rho_N\big( f \otimes \Psi(m) \big),
\end{align*}
i.e.\ $\Psi: (M, \rho_M) \to (N,\rho_N)$ is a module morphism over $A$. Conversely, if $\Psi$ is a module morphism over $A$, then 
\begin{align*}
\mu_{N,\KK} \circ (\id_N \otimes f)\circ \Big(\phi_N \circ \Psi\Big)(m) & = \rho_N(f \otimes \Psi(m)) \\
& = (-1)^{|f||\Psi|}\Psi(\rho_M(f \otimes m)) \\
& = (-1)^{|f||\Psi|}\Psi \circ \mu_{M,\KK}\circ (\id_M \otimes f) \circ \phi_M(m) \\
& = \mu_{N,\KK} \circ (\id_N \otimes f) \circ \Big((\Psi \otimes \id_C) \circ \phi_M \Big)(m)
\end{align*} 
for any $f \in \Hom(C,\KK)$, any $m \in M$. 
Thus, we have
$$
(\Psi \otimes \id_C) \circ \phi_M = \phi_N \circ \Psi,
$$
i.e.\ $\Psi$ is a comodule morphism over $C$. 
\end{proof}

\begin{example}\label{ex:SgcoMod}
Since $M=\cU\frakg \otimes_\tau  S(\frakg[1]) = (\cU\frakg \otimes S(\frakg[1]), d_{X})$ is a right dg comodule over $(S(\frakg[1]),\partial_\frakg)$, it is also a left dg module over $(S(\frakg[1])\dual,d_\frakg)$ by Proposition~\ref{prop:coModv.s.Mod}. 
More explicitly, the module structure $\rho:S(\frakg[1])\dual \otimes \cU\frakg \otimes S(\frakg[1]) \to \cU\frakg \otimes S(\frakg[1])$ is characterized by  
\begin{align*}
\rho\big(\xi \otimes (u \otimes x_1 \odot \cdots \odot x_n)\big) & = \mu_{M, \KK} \circ ( \id_M \otimes \pair{\xi}{\argument}) \circ (\id_{\cU\frakg} \otimes \Delta_S)(u \otimes x_1 \odot \cdots \odot x_n) \\
& = \sum_{i=1}^n (-1)^{i+1} \pair{\xi}{x_i} \ u \otimes x_1 \odot \cdots \widehat{x_i} \cdots \odot x_n \, ,
\end{align*}
for $\xi \in S^1 (\frakg[1])\dual$, $u \in \cU\frakg$, $x_1, \cdots, x_n \in \frakg[1]$. 
Furthermore, since $S(\frakg[1])\dual$ is graded commutative, the left action $\rho$ induce a right action $\tilde\rho$ on $M$: \begin{align*}
\tilde\rho\big( (u \otimes x_1 \odot \cdots \odot x_n) \otimes \xi \big) & = (-1)^{n} \rho\big(\xi \otimes (u \otimes x_1 \odot \cdots \odot x_n)\big) \\
& = \sum_{i=1}^n (-1)^{n-i} \pair{x_i}{\xi} \ u \otimes x_1 \odot \cdots \widehat{x_i} \cdots \odot x_n \, ,
\end{align*}  
which coincides with the contraction action $\Sgvaction$ defined in Section~\ref{sec:KellerTrip-LieAlg}.
\end{example}

\subsubsection*{Cogenerators of graded comodules}

Let $A$ be a $\KK$-algebra. A generator of a $A$-module $(N, \rho_N)$ can be considered as a vector space $W$, together with a linear map $\iota_{W}: W\rightarrow N$ such that the map
\[A\otimes W \xrightarrow{\id_{A}\otimes \iota_{W}} A\otimes N \xrightarrow{\rho_N} N\]
is surjective. Inspired by this point of view, we define cogenerator as follows.

\begin{definition}
Let $(C,\Delta, \epsilon)$ be a graded coalgebra, and $(M, \rho_M)$ be a graded comodule over $C$. 
A \textbf{cogenerator} of $M$ is a graded vector space $V$, together with a degree-preserving linear map $\pi_V:M\rightarrow V$, such that the composition 
\[
p_{V}:=(\pi_{V}\otimes \id_{C}) \circ \phi_{M}: M\xrightarrow{\phi_M}M\otimes C \xrightarrow{\pi_V\otimes \id_{C}} V\otimes C\]
is injective.
\end{definition}

Since the diagram
$$
\begin{tikzcd}
M \arrow[r, "\phi_M"'] \arrow[d, "\phi_{M}"'] \ar[rrr,bend left=15, "p_V"]  & M \otimes C \ar[rr," \pi_V \otimes \id_C"'] \ar[d,"\id_M \otimes \Delta"]& & V\otimes C \arrow[d, "\id_{V}\otimes \Delta"]\\
M\otimes C \arrow[r, "\phi_M \otimes \id_{C}"]\ar[rrr,bend right =15,"p_V \otimes \id"']& M \otimes C \otimes C \ar[rr,"\pi_V \otimes \id_C \otimes \id_C"] & & V\otimes C \otimes C \, ,
\end{tikzcd}
$$
commutes, the map $p_V:M\rightarrow V\otimes C$ is a morphism of comodules. 
A cogenerator $\pi_V$ is said to be \textbf{free} if $p_V$ is an isomorphism of comodules.

\begin{proposition}\label{prop:coGenerator}
Let $(M,\rho_M)$ and $(N, \rho_N)$ be graded comodules over $C$, and $\pi_{V}:N \rightarrow V$ be a cogenerator of $N$.  Then the pushforward map
\begin{equation*}
{\pi_{V}}_{\ast}:\coHom_C (M, N)\rightarrow \Hom(M,V)
\end{equation*}
is an embedding of graded vector spaces. Moreover, if $\pi_V$ is a free cogenerator, then ${\pi_{V}}_{\ast}$ is an isomorphism of graded vector spaces.
\end{proposition}
\begin{proof}
Let $\Psi_1,\Psi_2 \in \coHom_{C}(M,N)$ be comodule morphisms such that ${\pi_{V}}_{\ast}(\Psi_1)={\pi_{V}}_{\ast}(\Psi_2)$. 
Then, by \eqref{eq:coModMorAxiom}, one can show that
$$
p_V \circ \Psi_1 = p_V \circ \Psi_2.
$$
Since $p_V$ is injective, we have $\Psi_1 = \Psi_2$.

Assume $\pi_V$ is a free cogenerator. Since $p_V$ is an isomorphism of comodules, it suffices to verify the case
$$
N = V \otimes C, \qquad \pi_V = \mu_{V,\KK} \circ (\id_V \otimes \epsilon): V \otimes C \to V.
$$
For each  $f\in \Hom(M,V)$, it straightforward to show that the map
$$
\Psi_f:=(f\otimes \id_{C})\circ \phi_{M} : M \rightarrow V\otimes C,
$$ 
is a morphism of comodules such that $\pi_V \circ \Psi_f = f$. Thus, the proof is completed. 
\end{proof}

\begin{example}
Let $V$ and $W$ be a graded vector spaces. Then the pair $(V \otimes SW, \id_V \otimes \Delta)$ is a graded comodule over $SW$. The projection $\pr: V \otimes SW \onto V \otimes S^0W \cong V$ is a free cogenerator, because the composition 
\[
 V \otimes SW \xrightarrow{\id_{V} \otimes \Delta } V \otimes SW \otimes SW  \xrightarrow{\pr \otimes \id_{SW}} V \otimes SW \]
is the identity map. Thus, by Proposition~\ref{prop:coGenerator}, the pushforward map 
$$
{\pr}_{\ast}:\coHom_{SW} (V \otimes SW, V \otimes SW)\rightarrow \Hom(V \otimes SW,V)
$$ 
is an isomorphism of graded vector spaces. In fact, for $f \in \Hom(V \otimes SW,V)$, the composition 
\begin{equation}\label{eq:coGenModMor}
\Psi_f :V \otimes SW \xto{\id_V \otimes \Delta} V \otimes SW \otimes SW \xto{f \otimes \id_{SW}}  V \otimes SW, 
\end{equation}
is the comodule morphism such that $\pr_\ast(\Psi_f) = f$. 
\end{example}

\section{Hochschild complexes and tensor coalgebras}\label{appendix:HHcochains}

In this appendix, we recall Getzler's construction of Hochschild complexes in \cite{MR1261901} and show an isomorphism between the version here and our version in Section~\ref{sec:HHcxDGAlg}. 
For a dg algebra $A$, we consider Getzler's formulas as the natural formulas on $A[1]$, and ours are the formulas on $A$ obtained by Getzler's formulas composed with proper degree-shifting maps.

\subsection{Tensor coalgebras}

Let $V$ be a graded vector space. The \textbf{tensor coalgebra} $(TV, \Delta_T, \epsilon_T)$ over $V$ is the graded vector space $TV = \bigoplus_{n=0}^\infty V^{\otimes n}$ together with the counit $\epsilon_T = \pr : TV \onto V^{\otimes 0} = \KK$ and the coproduct $\Delta_T: TV \to TV \otimes TV$, 
\begin{equation*}
\begin{split}
\Delta_{T}(v_{1}\otimes \cdots\otimes v_{n})& = 1 \otimes (v_1 \otimes \cdots \otimes v_n) \\
& \quad +  \sum_{i=1}^{n-1}\big(v_1\otimes \cdots \otimes v_i\big)\otimes \big( v_{i+1}\otimes \cdots \otimes v_n \big)    +  (v_1 \otimes \cdots \otimes v_n) \otimes 1.
\end{split}
\end{equation*}

Let $\pr _{V}:TV \onto V$ be the canonical projection. By imitating the techniques of cogenerators in Section~\ref{sec:coGencoMod}, one can show that a coderivation $D \in \coDer(TV)$ is uniquely determined by $\pr _{V}\circ D \in \Hom(TV,V)$. In fact, for $q_{k}\in \Hom(V^{\otimes k},V)$, the map $q=\sum_{k}q_k \in \Hom(TV,V)$ determines a coderivation $\Tcoder q$ by the formulas
\[\Tcoder q|_{V^{\otimes n}}=\sum_{i+j+k=n}  \id^{\otimes i}\otimes q_{k} \otimes \id^{\otimes j}: V^{\otimes n} \to V^{\otimes n-k+1} .\] 
Since the space of coderivations with the graded commutator is a graded Lie algebra, the space $\Hom(TV,V)$ is also equipped with a Lie bracket 
\begin{equation*}
[f,g] := \pr _{V}\circ (\Tcoder f\circ \Tcoder g -(-1)^{\degree{f}\degree{g}}\Tcoder g \circ \Tcoder f)
\end{equation*}
where $f,g \in \Hom(TV,V)$ are arbitrary homogeneous maps.

\subsubsection*{Tensor coalgebra over a shifted dg algebra}

Let $A$ be a dg algebra equipped with differential $d_{A}$ and multiplication $\mu _{A}$. 
Following \cite{MR1261901}, we denote 
$$
m_1(\dsuspend a_1):= \dsuspend d_A(a_1), \qquad 
m_2( \dsuspend a_1, \dsuspend a_2):= (-1)^{|a_1|} \dsuspend \mu_A(a_1, a_2),  
$$
where $a_1, a_2 \in A$, and $\dsuspend: A \to A[1]$ is the degree-shifting map of degree $-1$. Let $\Tcoder m \in \coDer(T(A[1]))$ be the coderivation generated by $m := m_1+m_2 \in \Hom^1(T(A[1]),A[1])$. Since the dg algebra axioms $$
d_A^2=0, \qquad d_A \circ \mu_A = \mu_A\circ(d_A \otimes \id + \id \otimes d_A), \qquad \mu_A \circ(\mu_A\otimes \id) = \mu_A \circ (\id \otimes \mu_A)
$$ 
are equivalent to 
$$
\pr_{A[1]} \circ [\Tcoder m, \Tcoder m] \big|_{A[1]} =0, \qquad \pr_{A[1]} \circ [\Tcoder m, \Tcoder m] \big|_{A[1]^{\otimes 2}} =0, \qquad \pr_{A[1]} \circ [\Tcoder m, \Tcoder m] \big|_{A[1]^{\otimes 3}} =0,
$$
respectively, one has the equation 
$$
[\Tcoder m, \Tcoder m] =0.
$$ 
Therefore, we have

\begin{proposition}
The triple
$
\big(\Hom(T(A[1]),A[1]), [m, \argument], [\argument, \argument]\big)
$
is a dg Lie algebra. 
\end{proposition}

Let $M_1 = [m, \argument ]$, and 
$M_2:\Hom(T(A[1]),A[1])^{\otimes 2} \to \Hom(T(A[1]),A[1])$ 
be the operation 
$$
M_2(D_1,D_2):= m_2 \circ (D_1 \otimes D_2) \circ \Delta_T
$$ 
of degree one. One can show that 
\begin{equation}\label{eq:Ainfty[1]onHHcx}
[\Tcoder M, \Tcoder M] = 0,
\end{equation}
where $M= M_1 +M_2 \in \Hom\big(T\Hom(T(A[1]),A[1]),T\Hom(T(A[1]),A[1])\big)$. See \cite[Proposition~1.7]{MR1261901}.

Let 
$\widehat\hochschild: \Hom(T(A[1]),A) \to \Hom(T(A[1]),A)$ and 
$\widehat\cup: \Hom(T(A[1]),A)^{\otimes 2} \to \Hom(T(A[1]),A)$ be the unique maps satisfying the equations 
$$
M_1 \circ \dsuspend = \dsuspend \circ  \widehat\hochschild , \qquad 
M_2 \circ ( \dsuspend \otimes \dsuspend ) =  \dsuspend\circ \widehat\cup,  
$$
where $\dsuspend: \Hom(T(A[1]),A) \to \Hom(T(A[1]),A[1])$ is the degree-shifting map. By \eqref{eq:Ainfty[1]onHHcx}, we have the following 
 
\begin{proposition}
The triple $\big(\Hom(T(A[1]),A), \widehat\hochschild, \widehat\cup \big)$ is a dg algebra.  
\end{proposition}

\begin{remark}
More generally, if $A$ is an $A_\infty$ algebra, then so is $\Hom(T(A[1]),A)$. The construction is closely related to the braces operations on $\Hom(T(A[1]),A[1])$. See \cite{MR1328534,1994hep.th....3055G, MR1261901}. 
\end{remark}

\subsection{Hochschild cochains and coderivations}

Let $(A, d_{A}, \mu _{A})$ be a dg algebra. Let 
\[\dec:\Hom^{r}(A^{\otimes p},A) \rightarrow \Hom^{p+r}(A[1]^{\otimes p},A)\]
be the d\'ecalage map defined as
\[\dec(f)(\dsuspend a_{1}\otimes \cdots \otimes \dsuspend a_{p})= (-1)^{\sum_{i}(p-i)\degree{a_{i}}}  f(a_{1}\otimes \cdots \otimes a_{p})\]
for $f\in \Hom^{r}(A^{\otimes p},A)$ and $a_{1},\cdots,a_{n}\in A$. 
In other words, $\dec(f) \in \Hom^{p+r}(A[1]^{\otimes p},A)$ is the unique map such that the diagram
$$
\begin{tikzcd}
A[1]^{\otimes p} \ar[r,"\dec(f)"] & A \\
A^{\otimes p} \ar[ru,"f"'] \ar[u,"\dsuspend^{\otimes p}"] &
\end{tikzcd}
$$
commutes. 
Note that $m_1 = \dsuspend \circ \dec(d_A)$ and $m_2 = \dsuspend\circ \dec(\mu_A)$.

\begin{proposition}
The map 
$$
\dsuspend\circ \dec\circ\dsuspend\inv: (\grHHcx^\bullet(A)[1],\gerstenhaber\argument\argument) \to (\Hom(T(A[1]),A[1]), [\argument,\argument])
$$
is an embedding of graded Lie algebras. In particular, 
$$
\dec  \circ (\hochschild + \partial_A) = \widehat{\hochschild} \circ \dec. 
$$  
\end{proposition}
\begin{proof}
Let $f \in \HHcx^{p_1,r_1}(A)$ and $g \in \HHcx^{p_2,r_2}(A)$. Since 
\begin{align*}
& (\dsuspend \dec  f) \circ \big(\id^{\otimes i-1} \otimes (\dsuspend \dec  g ) \otimes \id^{\otimes p_1-i}\big) \circ \dsuspend^{\otimes p_1+p_2-1}  \\
& \qquad = (-1)^{(i-1)(p_2+r_2-1)}(\dsuspend \dec  f) \circ \big(\dsuspend^{\otimes i-1} \otimes (\dsuspend \dec  g  \circ \dsuspend^{\otimes p_2}) \otimes \dsuspend^{\otimes p_1-i}\big) \\
& \qquad = (-1)^{(i-1)(p_2+r_2-1)+(p_1-i)r_2}(\dsuspend \dec  f) \circ \dsuspend^{\otimes p_1} \circ (\id^{\otimes i-1} \otimes  g \otimes \id^{\otimes p_1-i}) \\
& \qquad = (-1)^{(i-1)(p_2-1)+(p_1-1)r_2}\, \dsuspend  \circ f \circ (\id^{\otimes i-1} \otimes  g \otimes \id^{\otimes p_1-i}),
\end{align*}
it follows from \eqref{eq:GerBracket} that 
$$
[\dsuspend \dec  f, \dsuspend \dec  g]  \circ \dsuspend^{\otimes p_{1}+p_{2}-1}  = \dsuspend \circ \gerstenhaber f g ,
$$
which implies the assertion.
\end{proof}

\begin{proposition}
The map 
$$
\dec: (\grHHcx^\bullet(A),\hochschild+\partial_A ,\cup \, ) \to (\Hom(T(A[1]),A),\widehat\hochschild  ,\widehat\cup \, )
$$
is an embedding of dg algebras. 
\end{proposition}
\begin{proof}
Let $f \in \HHcx^{p_1,r_1}(A)$ and $g \in \HHcx^{p_2,r_2}(A)$. Since 
\begin{align*}
\dsuspend (\dec f \, \widehat{\cup} \dec g) & = M_2 \circ (\dsuspend \otimes \dsuspend)\circ (\dec f \otimes \dec g) \\
& =  m_2\circ (\dsuspend \otimes \dsuspend) \circ (\dec f \otimes \dec g) \circ \Delta_T \\
& =   \dsuspend \circ \mu_A \circ (\dec f \otimes \dec g) \circ \Delta_T,
\end{align*}
we have 
\begin{align*}
(\dec f \, \widehat{\cup} \dec g)\circ (\dsuspend^{\otimes p_1+p_2})& = \mu_A \circ (\dec f \otimes \dec g) \circ \Delta_{TA[1]} \circ (\dsuspend^{\otimes p_1+p_2})\\
 & = \mu_A \circ (\dec f \otimes \dec g) \circ (\dsuspend^{\otimes p_1} \otimes \dsuspend^{\otimes p_2})\circ \Delta_{TA} \\
&  = (-1)^{(p_2+r_2)p_1} \mu_A \circ(f \otimes g) \circ \Delta_{TA},
\end{align*}
where $\Delta_{TA[1]}$ is the coproduct on $TA[1]$, and $\Delta_{TA}$ is the coproduct on $TA$. Thus, by comparing the above formula with \eqref{eq:CupProd}, we conclude 
$$
(\dec f \, \widehat{\cup} \dec g)\circ (\dsuspend^{\otimes p_1+p_2}) = f \cup g,
$$
which implies the assertion.
\end{proof}

\bibliography{ref_KellerDuflo}

\end{document}